\documentclass[10pt,a4paper]{article} 
\usepackage[shortlabels]{enumitem}
\usepackage{graphicx}
\usepackage{enumitem}
\usepackage{float}
\usepackage{amsthm,amssymb,mathrsfs,amsmath}
\usepackage{comment}
\usepackage{xcolor}
\usepackage[a4paper, bindingoffset=0.5in, left=0.5in,right=0.5in, top=1in, bottom=1in, footskip=.25in]{geometry}
\usepackage[colorlinks,citecolor=blue,urlcolor=blue,bookmarks=false,hypertexnames=true]{hyperref} 

\newtheorem{theorem}{Theorem}[section]
\numberwithin{figure}{section}

\newtheorem{definition}{Definition}[section]

\begin{document}
\title{A generalized novel approach based on orthonormal polynomial wavelets with an application to Lane-Emden equation}
\author{Amit K. Verma$^a$\\\small{\textit{$^{a}$Department of Mathematics, IIT Patna, Patna $801106$, Bihar, India.}}\\ Diksha Tiwari$^b$\footnote{$^a$akverma@iitp.ac.in,$^b$dikshatiwari227@gmail.com}
	\\\small{\textit{$^b$Faculty of Mathematics, University of Vienna, Austria.}}\\  Carlo Cattani$^c$\thanks{cattani@unitus.it}\\\small{\it{$^c$Engineering School (DEIM), University of Tuscia,}}\\\small{\it{
Largo dell'Universit\ a, 01100 Viterbo, Italy.}}}
\date{\today}

\maketitle

\begin{abstract}
Capturing solution near the singular point of any nonlinear SBVPs is challenging because coefficients involved in the differential equation blow up near singularities. In this article, we aim to construct a general method based on orthogonal polynomials as wavelets.  We discuss multiresolution analysis for wavelets generated by orthogonal polynomials, e.g., Hermite, Legendre, Chebyshev, Laguerre, and Gegenbauer. Then we use these wavelets for solving nonlinear SBVPs. These wavelets can deal with singularities easily and efficiently. To deal with the nonlinearity, we use both Newton's quasilinearization and the Newton-Raphson method. To show the importance and accuracy of the proposed methods, we solve the Lane-Emden type of problems and compare the computed solutions with the known solutions. As the resolution is increased the computed solutions converge to exact solutions or known solutions. We observe that the proposed technique performs well on a class of Lane-Emden type BVPs. As the paper deals with singularity, non-linearity significantly and different wavelets are used to compare the results. 
\end{abstract}
\section{Introduction}\label{P3_sec1}
The solution of singular differential equations shows unusual behavior near the singular points, sometimes it is bounded, sometimes unbounded, sometimes it may oscillate or sometimes it is peculiar in some other manner. This behavior and its occurrence in different areas of science and engineering make SBVPs very interesting for researchers.
Consider the following class of nonlinear singular differential equations
\begin{equation}\label{P3_eq1}
ty''(t)+ky'(t)+t{f(t,y(t),t^ky'(t))}=0,\quad\quad\quad 0<t\leq 1,
\end{equation}
 subject to the following boundary conditions
\begin{eqnarray}
\label{P3_1c}&& y'(0)=\alpha,\quad ay(1)+by'(1)=\beta.
\end{eqnarray}
Several real life problems when modelled gives rise to nonlinear partial differential equation \[\nabla^2 u(P)=f(P,u(P))\] and if one is interested in planar ($k=0$), cylindrical ($k=1$) or spherical ($k=2$) symmetry then one arrives at \eqref{P3_eq1}. For various values of $k$ and different type of source functions $f(t,y(t),t^ky'(t))$,  the BVP defined by \eqref{P3_eq1}-\eqref{P3_1c}  occur in different areas of science and engineering. Some of them we mention as given below:\\

\noindent a) when $k=2$ and $f(t,y,t^ky')=y^n$, then it occurs in the study of stellar structure \cite{CS1967}, \\
b) when $k=0,1,2$ and $f(t,y,t^ky')=e^y$, then it occurs in thermal explosion \cite{CHAMBRE1952}, \\
c) when $k=2$, $f(t,y,t^ky')=e^{-y}$, then it occurs in thermal distribution in the human head \cite{RA1986}, \\
d) when $k=3$, $f(t,y,t^ky')=\frac{1}{8y^2}-\frac{\tilde{a}}{y}-\tilde{b}t^{2\tilde{c}-4},$ then it occurs in rotationally symmetric shallow membrane caps \cite{RW1989,JV1998}. \\

For existence-uniqueness of SBVPs \eqref{P3_eq1}-\eqref{P3_1c} the reader is suggested to refer \cite{pandey2008existence,Pandey2008existence1,pandey2009note} and the references cited in a recent review article \cite{akvsbvpreview2020}. Some very efficient numerical schemes based on finite difference can be found in \cite{Chawla1988,rkpaks2009,skcoam2020} and the references therein. 

Wavelet methods arise as one of trending methods to solve differential equations. Due to its properties like smoothness, well-localization, admissibility and orthonormality methods based on wavelets are most preferred by scientists and engineers. In \cite{Carlo2001} author proposed numerical approximation of differential operators using Haar wavelets and their spline-derivatives. A method based on Haar Wavelets are used for solving generalized Lane-Emden equations in \cite{RC2013}. Authors in \cite{Mittal2017} developed a novel
algorithm based on scale-3 Haar wavelets, applied it on Burgers' equation and did sensitivity analysis of shock waves. In \cite{SC2016} method based on Haar wavelets are used for solving non-linear singular initial value problems and in \cite{akvdt2019} Haar wavelets coupled with quasilinearization is used to solve class of Lane-Emden equation at higher resolution. In \cite{rsg2019,rsgjshgag2019} Haar wavelets are efficiently used to solve SBVPs arising in various real life problems. Haar wavelet along with quasilinearization is used to solve nonlinear BVPs \cite{VH2011, VF2001}.  A new higher order Haar wavelet method has been developed for solving differential and integro-differential equations in \cite{MAJAK2019AIP,MAJAK2018CS}.

The construction of wavelets based on orthogonal polynomials is more recent. Since it is easy to generate an orthonormal basis of $L^2(\mathbb{R})$ with help of orthogonal polynomials, so wavelets based on orthogonal polynomials are very popular nowadays. Different orthonormal wavelets are defined by using Legendre polynomials, Chebyshev polynomials, Hermite polynomials, Laguerre polynomials, Gegenbauer polynomials and researchers used these wavelets for solving various classes of differential equations and other related problems. We list some important results in which various orthonormal polynomials are used to compute the solutions of various classes of differential equations.
\begin{enumerate}[(i)]
\item Legendre wavelet \cite{KHELLAT2006} is used to solve ordinary differential equation \cite{MOH2011} and q-difference equations \cite{Vijesh2016},  nonlinear system of two-dimensional integral equations \cite{Maleknejad2019}. 
\item Chebyshev wavelet is used to solve singular BVPs \cite{RR2015} and Sine-Gordon equation \cite{Antony2017}.
\item Hermite wavelet is used to solve singular differential equations \cite{Usman2013}. 
\item Laguerre wavelet is used to solve linear and non-linear singular BVPs  \cite{Zhou2016}. 
\item Gegenbauer wavelet is used to solve fractional differential equation \cite{Mujeeb2015}. 
\end{enumerate}
In this article, we construct wavelet methods based on orthogonal polynomials as wavelets. They are referred as Chebyshev wavelet Newton approach (ChWNA), Gegenbauer wavelet Newton approach (GeWNA), Legendre wavelet Newton approach (LeWNA), Laguerre wavelet Newton Approach (LaWNA), Hermite wavelet Newton approach (HeWNA), Chebyshev wavelet quasilinearization approach (ChWQA), Gebenauer wavelet quasilinearization approach (GeWQA), Legendre wavelet quasilinearization approach (LeWQA), Laguerre wavelet quasilinearization approach (LaWQA) and Hermite wavelet quasiliniearization approach (HeWQA). Such an approach has not been explored by researchers in the existing literature. We apply our novel approach to a class of nonlinear singular BVPs which are also referred as Lane-Emden equations. This may further be explored on various real life problems governed by differential equations and fractional nonlinear differential equations (\cite{Sunil2014, Rashidi2014, Ranbir2020, Mugisha2020, Bessem2020, Manafian2020, Hossein2020, Fatima2020, Aziz2020}).

This paper is organized as follows. In section \ref{P3_sec2} we discuss properties of orthogonal wavelets defined by using Legendre, Hermite, Chebyshev, Laguerre and Gegenbauer polynomials. Section \ref{P3_sec6}  and \ref{P3_sec7} discuss computational aspects of orthonormal polynomial wavelets and methods based on these wavelets, respectively.  We follow two approaches, wavelet quasilinearization approach and wavelet Newton approach. Section \ref{P3_sec8} deals with convergence analysis of the methods based on wavelet Newton approach. Finally section \ref{P3_sec9} deals with numerical illustrations and section \ref{P3_conclusions} deals with conclusion.
\section{Preliminary} \label{P3_sec2}
Most of the orthonormal wavelets' bases are associated with an MRA but it is possible to construct an orthonormal basis for $L^2(\mathbb{R})$ which is not derivable from an MRA. Here, we give a formal definition of MRA (\cite{BN2009, ID1992, MCPLAW2012}).
The sequence of wavelet subspaces $W_j$ of $L^2(\mathbb{R})$, are such that $V_j\perp W_j$, for all $j$ and $V_{j+1} = V_j\oplus W_j$. The closure of $\oplus_{j\in \mathbb{Z}} W_j$ is dense in $L^2(\mathbb{R})$ with respect to the $L^2$ norm.

The definition of MRA enables the sequence $\left\{2^{j/2}\varphi(2^{-j}t-k)\right\}_{k\in \mathbb{Z}}$ to form an orthonormal basis for $V_j$. Let $\psi(t)$ be the mother wavelet, then, $$\psi(t)=\sum_{k\in \mathbb{Z}}a(k)\varphi(2t-k),$$ and under certain conditions $\psi_{j,k}(t)=2^{-j/2}\psi(2^{-j}t-k)$ forms an orthonormal basis of $L^2(\mathbb{R})$. 

The basic tenet of MRA is that whenever a collection of closed subspaces satisfies assumptions of MRA, then there exists an orthonormal wavelet basis such that, for all $f\in L^2(\mathbb{R})$,
$$P_{j-1}f=P_jf+\sum_{k\in \mathbb{Z}}\langle~f,\psi_{j,k}\rangle\psi_{j,k},$$
where $P_j$ is the orthogonal projection onto $V_j$. To establish this, it is enough to show that $\psi\in W_0$ such that
the $\psi(\cdot-k)$ constitute an orthonormal basis for $W_0$. 

The most promising class of scaling functions ($\varphi$) are those, that have compact support and continuity makes it even better. If any scaling function has both properties then the associated decomposition and reconstruction algorithms are computationally much faster and best suited for analysing and reconstructing signals.

\subsection{Orthogonal Polynomial Wavelet}
Let $O_{m}(t)$ be an orthonormal polynomial \cite{ID1992} of degree $m$ which is defined on $[a,b]$. Let $O_m(t)$ satisfy the following orthonormality condition:
\begin{eqnarray}
\label{OrthoGen} \int_a^bw(t)O_m(t)O_n(t)dt=K \delta_{mn},
\end{eqnarray}
with respect to weight function $w(t)$, $K$ is some real number.  We define the orthonormal polynomial wavelet on the interval $[0,1)$ as follows
\begin{equation}\label{Pu}
\psi_{n,m}(t)=
v_n2^{k/2}O_{m}(2^kt-\hat{n})\chi_{[\frac{\hat{n}-1}{2^k},\frac{\hat{n}+1}{2^k})},
\end{equation}
where $k=1,2,\hdots$ is the level of resolution, $n=1,2,\hdots, 2^{k-1},\hspace{0.25cm} \hat{n}=2n-1$ is translation parameter and $m=0,1,2,\hdots,M-1$ is the degree of the polynomial. The coefficient $v_n$ is a real number and is kept here to take care the orthonormality. It is easy to check that $\psi_{n,m}(t)$ forms an orthonormal basis for $L^{2}(\mathbb{R})$. 

Now we study various orthonormal wavelets and discuss their applications in the upcoming sections.
\subsubsection{Chebyshev Wavelet}\label{p3_sub1}
Chebyshev polynomials \cite{BIAZAR2012608} are defined on the interval $[-1,1]$ with help of the recurrence formula:
\begin{eqnarray*}
T_{0}(t)=1,~T_{1}(t)=t, \quad~T_{m+1}(t)= 2t~T_{m}(t)-T_{m-1}(t), ~m=1,2,3,\cdots.
\end{eqnarray*}
These polynomials are orthonormal with respect to the weight function $\frac{1}{\sqrt{1-t^2}}$ on $[-1,1]$.

Chebyshev wavelets are defined on the interval $[0,1]$ as follows
\begin{equation} \label{P3_Def1}
\psi_{n,m}(t)=
2^{k/2}\bar{T}_{m}(2^kt-\hat{n})\chi_{[\frac{\hat{n}-1}{2^k},\frac{\hat{n}+1}{2^k})},
\end{equation}
where
\begin{equation}
\bar{T}_{m}(t)= \begin{cases}\frac{1}{\sqrt{\pi}},\quad m=0,\\
\frac{1}{\sqrt{\pi}}T_{m}(t), \quad m >0,
\end{cases}
\end{equation}
where $k=1,2,\hdots$ is the level of resolution, $n=1,2,\hdots, 2^{k-1},\hspace{0.25cm} \hat{n}=2n-1$ is the translation parameter, $m=0,1,2,\hdots,M-1$ is the degree of Chebyshev polynomial.
\subsubsection{Hermite Wavelet}\label{p3_sub2}
Hermite polynomials \cite{Saeed2014} are defined on the interval $(-\infty,\infty)$ with help of the recurrence formula:
\begin{eqnarray*}
H_{0}(t)=1,~H_{1}(t)=2t,\quad~ H_{m+1}(t)= 2t~H_{m}(t)-2mH_{m-1}(t),~m=1,2,3,\cdots.
\end{eqnarray*}
These polynomials are orthonormal with respect to the weight function $e^{-t^2}$.

Hermite wavelets are defined on the interval $[0,1]$ as follows:
\begin{equation} \label{P3_Def2}
\psi_{n,m}(t)=
2^{k/2}\frac{1}{\sqrt{n!2^{n}\sqrt{\pi}}}H_{m}(2^kt-\hat{n})\chi_{[\frac{\hat{n}-1}{2^k},\frac{\hat{n}+1}{2^k})},
\end{equation}
where $k=1,2,\hdots$ is the level of resolution, $n=1,2,\hdots, 2^{k-1},\hspace{0.25cm} \hat{n}=2n-1$ is the translation parameter, $m=0,1,2,\hdots,M-1$ is the degree of Hermite polynomial.

\subsubsection{Laguerre Wavelet}\label{p3_sub5}
Laguerre polynomials \cite{IQBAL201550}  are defined on the interval $(-\infty,\infty)$ and with help of the recurrence formula:
\begin{eqnarray*}
L_{0}(t)=1,~L_{1}(t)=1-t,\quad~(m+1)L_{m+1}(t)= (2m+1-t)L_{m}(t)-mL_{m-1}(t),~m=1,2,3,\cdots.
\end{eqnarray*}
These polynomials are orthonormal with respect to the weight function $e^{-t}$.

Laguerre wavelets are defined on the interval $[0,1]$ as follows
\begin{equation} \label{P3_Def3}
\psi_{n,m}(t)=
2^{k/2}\frac{1}{n!}L_{m}(2^kt-\hat{n})\chi_{[\frac{\hat{n}-1}{2^k},\frac{\hat{n}+1}{2^k})},
\end{equation}
where $k=1,2,\hdots$ is the level of resolution, $n=1,2,\hdots, 2^{k-1},\hspace{0.25cm} \hat{n}=2n-1$ is the translation parameter, $m=0,1,2,\hdots,M-1$ is the degree of Laguerre polynomial.

\subsubsection{Legendre Wavelet}\label{p3_sub3}
Legendre polynomials \cite{MOH2011} are defined on the interval $[-1,1]$ with help of the recurrence formula:
\begin{eqnarray*}
P_{0}(t)=1,~P_{1}(t)=t,\quad~(m+1)P_{m+1}(t)= (2m+1)tP_{m}(t)-mP_{m-1}(t),~m=1,2,3,\cdots.
\end{eqnarray*}
These polynomials are orthonormal.

Legendre wavelets are defined on the interval $[0,1]$ as follows
\begin{equation} \label{P3_Def4}
\psi_{n,m}(t)=
2^{k/2}\sqrt{\left(n+\frac{1}{2}\right)}P_{m}(2^kt-\hat{n})\chi_{[\frac{\hat{n}-1}{2^k},\frac{\hat{n}+1}{2^k})},
\end{equation}
where $k=1,2,\hdots$ is the level of resolution, $n=1,2,\hdots, 2^{k-1},\hspace{0.25cm} \hat{n}=2n-1$ is the translation parameter, $m=0,1,2,\hdots,M-1$ is the degree of Legendre polynomial.
\subsubsection{Gegenbauer Wavelet}\label{p3_sub4}
Gegenbauer polynomials \cite{Mujeeb2015} are defined on the interval $[-1,1]$ and can be defined with help of the recurrence formula:

\begin{eqnarray*}
C_{0}^{\alpha}(t)=1,~C_{1}^{\alpha}(t)=2\alpha t,\quad~C_{m}^{\alpha}(t)= \frac{1}{m}\left[2t(n+\alpha-1)C_{m-1}^{\alpha}(t)-(n+2\alpha-2)C_{m-1}^{\alpha}(t)\right],~m=1,2,3,\cdots.
\end{eqnarray*}
These polynomials are orthogonal with respect to the weight function $(1-t^2)^{\alpha-\frac{1}{2}}$.

For $\alpha > -\frac{1}{2}$, Gegenbauer wavelets are defined on the interval $[0,1]$ as follows
\begin{equation} \label{P3_Def5}
\psi_{n,m}(t)=
2^{k/2}\frac{1}{\sqrt{\alpha}}C_{m}^{\alpha}(2^kt-\hat{n})\chi_{[\frac{\hat{n}-1}{2^k},\frac{\hat{n}+1}{2^k})},
\end{equation}
where $k=1,2,\hdots$ is the level of resolution, $n=1,2,\hdots, 2^{k-1},\hspace{0.25cm} \hat{n}=2n-1$ is the translation parameter, $m=0,1,2,\hdots,M-1$ is the degree of Gegenbauer polynomial.

\section{Computation with Orthogonal Polynomial Wavelets}\label{P3_sec6}
\subsection{Approximation of a $L^2$ Function}
A function $f(t)$ defined on $L^2[0,1]$ can be approximated with any of the above orthogonal polynomial wavelet in the following manner
\begin{equation}\label{P3_1}
f(t)= \sum^{\infty}_{n=1}\sum^{\infty}_{m=0}c_{nm}\psi_{nm}(t).
\end{equation}
Now, by truncating \eqref{P3_1}, we define
\begin{equation}\label{P3_2}
f(t)\simeq \sum^{2^{k}-1}_{n=1}\sum^{M-1}_{m=0}c_{nm}\psi_{nm}(t)=c^{T}\psi(t),
\end{equation}
where $k=1,2,\hdots$ is level of resolution, $n=1,2,\hdots, 2^{k-1},\hspace{0.25cm} \hat{n}=2n-1$ is the translation parameter, $m=0,1,2,\hdots,M-1$ is the degree of orthogonal polynomial, and  $\psi(t)$ is $2^{k-1}M\times 1$ matrix given as:
\begin{equation*}\label{P3_3}
\psi(t)= [\psi_{1,0}(t),\dots,\psi_{1,M-1}(t),\psi_{2,0}(t),\dots,\psi_{2,M-1}(t),\dots,\psi_{2^{k-1},0}(t),\dots,\psi_{2^{k-1},M-1}(t)]^T.
\end{equation*}
Here $c$ is $2^{k-1}M\times 1$ matrix and $M$ is the degree of the orthogonal polynomial. Given $f$ the entries of $c$ can be computed as 
\begin{equation}\label{P3_4}
c_{ij}=\int^{1}_{0}w(t) \psi_{ij}(t)f(t)\ dt.
\end{equation}
\subsection{Integration of Orthogonal Polynomial Wavelets}\label{p3_sub6}
As suggested in \cite{Gupta2015}, $\nu$-th order integration of $\psi(t)$ can also be approximated as
\begin{multline*}\label{P3_5}
\int^{t}_{0}\int^{t}_{0}\dots\int^{t}_{0}\psi(\tau)d\tau\\
\simeq[P^\nu\psi_{1,0}(t),\dots,P^\nu\psi_{1,M-1}(t),P^\nu\psi_{2,0}(t),\dots,P^\nu\psi_{2,M-1}(t),\dots,P^\nu\psi_{2^{k-1},0}(t),\dots,P^\nu\psi_{2^{k-1},M-1}(t)]^T,
\end{multline*}
where
\begin{equation}
 P^{\nu}\psi_{n,m}(t)= v_n2^{k/2}P^{\nu}O_{m}(2^kt-\hat{n})\chi_{[\frac{\hat{n}-1}{2^k},\frac{\hat{n}+1}{2^k})},
\end{equation}
and $v_m$ is an appropriate orthonormality constant.

Note: Integral operator $P^{\nu}(\nu>0)$ of a function $f(t)$ is defined as
\begin{equation*}\label{P3_6}
P^{\nu}f(t)= \frac{1}{\nu!}\int_{0}^{t}(t-s)^{\nu-1}f(s)ds.
\end{equation*}
\subsection{ Wavelets Collocation Method}\label{P3_sub7}
For application of the above orthogonal polynomial wavelets to the ordinary differential equations, discretization of $[0,1]$ is required. Here we use collocation method for discretization of  the interval $[0,1]$. Hence, we may define the mesh points as follows:
\begin{eqnarray}
\label{P3_91}\bar t_{l} = l \Delta t,\quad\quad\quad l = 0,1,\cdots,M-1.
\end{eqnarray}
For the collocation points we use the following relationship
\begin{eqnarray}
\label{P3_92}t_{l} = 0.5(\bar t_{l-1} + \bar t_{l}), \quad\quad\quad l = 1,\cdots,M-1.
\end{eqnarray}
For computation purpose we take $k=1$ and hence \eqref{P3_2} takes the following form
\begin{equation}\label{P3_81}
f(t)\simeq \sum^{M-1}_{m=0}c_{1m}\psi_{1m}(t).
\end{equation}
Now we replace $t$ by $t_l$ and solve the resulting system of linear equations to get the wavelet coefficient $c_{1m}$, hence approximate expression of the function may be obtained by the above equation. 
\section{Method of Solution}\label{P3_sec7}
In this section, general solution method is presented. Later for illustration we use different orthogonal polynomial wavelets on Lane-Emden equations and compute the numerical solutions.
\subsection{Wavelets Quasilinearization Approach (WQA)}\label{P3_sub8}
In this method, we use quasilinearization to linearize the SBVPs then we use the method of collocation for discretization and orthogonal wavelet methods for the computation of numerical solutions. We consider differential equation \eqref{P3_eq1} with boundary conditions \eqref{P3_1c}. Quasilinearizing equation \eqref{P3_eq1}, we get the following form of equation
\begin{subequations}\label{P3_7}
\begin{eqnarray}
\label{P3_22a} Ly_{r+1} =- y''_{r+1}(t)-\frac{k}{t}y'_{r+1}(t)=f(t,y_{r}(t))+(y_{r+1}-y_{r})(f_{y}(t,y_{r}(t)),
\end{eqnarray}
subject to the following boundary conditions,
\begin{eqnarray}
\label{P3_22d}&&y'_{r+1}(0)=\alpha,\quad\quad a{y_{r+1}(1)}+by'_{r+1}(1)=\beta.
\end{eqnarray}
Now we use the orthogonal polynomial wavelets method  \cite{CH1997} and assume 
\begin{eqnarray}
\label{P3_22e}y''_{r+1}(t)=\sum^{M-1}_{m=0}c_{1m}\psi_{1m}(t).
\end{eqnarray}
Integrating twice we get the following two equations:
\begin{eqnarray}
&&\label{P3_22f}y'_{r+1}(t)=\sum^{M-1}_{m=0}c_{1m}P\psi_{1m}(t)+y'_{r+1}(0),\\
&&\label{P3_22g}y_{r+1}(t)=\sum^{M-1}_{m=0}c_{1m}P^2\psi_{1m}(t)+ty'_{r+1}(0)+y_{r+1}(0).
\end{eqnarray}
\end{subequations}
\subsubsection{Treatment of the Boundary Conditions}\label{P3_subsubsec1}
Now replacing $t$ by $1$ in equation \eqref{P3_22f} and \eqref{P3_22g}, we get
\begin{eqnarray}\label{P3_17}
&&y'_{r+1}(1)=\sum^{M-1}_{m=0}c_{1m}P\psi_{1m}(1)+y'_{r+1}(0),\\
\label{P3_18}&&y_{r+1}(1)=\sum^{M-1}_{m=0}c_{1m}P^2\psi_{1m}(1)+y'_{r+1}(0)+y_{r+1}(0).
\end{eqnarray}
Putting these values in $ay_{r+1}(1)+by'_{r+1}(1)=\beta$ and solving for $y_{r+1}(0)$, we have
\begin{eqnarray*}\label{P3_19}
y_{r+1}(0)=\frac{1}{a}\left(\beta - ay'_{r+1}(0)-a\sum^{M-1}_{m=0}c_{1m}P^2\psi_{1m}(1)-b\left(\sum^{M-1}_{m=0}c_{1m}P\psi_{1m}(1)+y'_{r+1}(0)\right)\right).
\end{eqnarray*}
Hence from equation \eqref{P3_22g} we get
\begin{eqnarray}\label{P3_20}
y_{r+1}(t)=\sum^{M-1}_{m=0}c_{1m}P^2\psi_{1m}(t)+ty'_{r+1}(0)+\frac{1}{a}\left(\beta - ay'_{r+1}(0)-a\sum^{M-1}_{m=0}c_{1m}P^2\psi_{1m}(1)-b\left(\sum^{M-1}_{m=0}c_{1m}P\psi_{1m}(1)+y'_{r+1}(0)\right)\right).
\end{eqnarray}
Now we put values of $y_{r+1}(0)$ and $y_{r+1}(1)$ in equation \eqref{P3_22f} and \eqref{P3_20} we get,
\begin{eqnarray}
&&\label{P3_21}y'_{r+1}(t)=\alpha +\sum^{M-1}_{m=0}c_{1m}P\psi_{1m}(t),\\
&&\label{P3_22}y_{r+1}(t)=\frac{\beta}{a}+\left(t-1-\frac{b}{a}\right)\alpha+\sum^{M-1}_{m=0}c_{1m}\left(P^2\psi_{1m}(t)-P^2\psi_{1m}(1)-\frac{b}{a}P\psi_{1m}(1)\right).
\end{eqnarray}
Finally, we put values of $y''_{r+1}$, $y'_{r+1}$ and $y_{r+1}$ in the linearized differential equation \eqref{P3_22a}. Now we discretize the final equation with the collocation method and then solve the resultant linear system of equations by using a suitable initial guess $y_{0}(t)$. If the function values are known at one of the boundary points then choosing it as an initial guess is better. Finally, we get the required values of $y(t)$ at different collocation points for a given spatial points. 
\subsection{Wavelet Newton Approach (WNA)}\label{P3_sub9}
In this approach, we use the method of collocation for discretization and then we use an orthogonal polynomial wavelet for further computation. Finally, the Newton-Raphson method is used to solve the resulting nonlinear system of equation.

We consider differential equation \eqref{P3_eq1} with boundary conditions of the type  \eqref{P3_1c}. 
\subsubsection{Treatment of the Boundary Conditions}\label{P3_subsubsec2}
By the analysis similar to sub section \ref{P3_subsubsec1}, we arrive at the following set of equations
\begin{eqnarray}
&&y'(t)=\alpha +\sum^{M-1}_{m=0}c_{1m}P\psi_{1m}(t),\\
&&\label{P3_38}
y(t)=\frac{\beta}{a}+\left(t-1-\frac{b}{a}\right)\alpha+\sum^{M-1}_{m=0}c_{1m}\left(P^2\psi_{1m}(t)-P^2\psi_{1m}(1)-\frac{b}{a}P\psi_{1m}(1)\right).
\end{eqnarray}
Now we put the values of $y(t)$, $y'(t)$ and $y''(t)$ in \eqref{P3_eq1} and discretize the resulting equation with collocation method and solve the resultant nonlinear system with Newton-Raphson method for $c_{1m},~m=0,1,\dots,M-1$. Then by substituting value of $c_{1m}, m=0,1,\dots,M-1$, in \eqref{P3_38}, we will get the value of $y(t)$ at different collocation points.
\section{Convergence of WNA method based on orthogonal polynomials}\label{P3_sec8}
We will consider the following lemma for proving convergence of  wavelet Newton approach based on orthogonal polynomial, for quasilinearization approach it follows accordingly (see \cite{akvdt2019}).
Let us consider $2$-$nd$ order ordinary differential equation in general form
\begin{equation*}\label{P3_39}
G(t,y,y',y'')=0.
\end{equation*}
Now we have in WNA method
\begin{equation}\label{P3_con1}
f(t)=y''(t)= \sum^{\infty}_{n=1}\sum^{\infty}_{m=0}c_{nm}\psi_{nm}(t).
\end{equation}
Integrating this relation two times we have 
\begin{equation}\label{P3_con2}
y(t)= \sum^{\infty}_{n=1}\sum^{\infty}_{m=0}c_{nm}P^2\psi_{nm}(t)+B_{T}(t),
\end{equation}
where $B_{T}(t)$ stands for boundary term.

\begin{theorem}
	Let us assume that, $f(t)=\frac{d^2y}{dt^2}\in L^2[0,1]$ is a continuous function defined on $[0,1]$. Let us consider that $f(t)$ is bounded, i.e.,
	\begin{equation}\label{P3_con3}
	\forall  t \in [0,1] \quad \quad \exists \quad \eta :\left|\frac{d^2y}{dt^2}\right| \leq \eta.
	\end{equation}
	Then method based on Wavelet Newton Approach (WNA) converges.
\end{theorem}
\begin{proof}
	In \eqref{P3_con2}, truncating the expansion we have,
	\begin{eqnarray}\label{P3_con4}
	y^{k,M}(t)= \sum^{2^k-1}_{n=1}\sum^{M-1}_{m=0}c_{nm}P^2\psi_{nm}(t)+B_{T}(t).
	\end{eqnarray}

	So error $E_{k,M}$  can be expressed as
	\begin{eqnarray}\label{P3_con5}
	||E_{k,M}||_{2}=||y(t)-y^{k,M}(t)||_{2}=\left|\left|\sum^{\infty}_{n=2^k}\sum^{\infty}_{m=M}c_{nm}P^2\psi_{nm}(t)\right|\right|_{2}.
	\end{eqnarray}
	Expanding the $L^{2}$ norm, we have
	\begin{eqnarray}\label{P3_con6}
	&&\nonumber	||E_{k,M}||_2^2=\int_{0}^{1}\left(\sum^{\infty}_{n=2^k}\sum^{\infty}_{m=M}c_{nm}P^2\psi_{nm}(t)\right)^2dt,\\
	\label{P3_con8}
	&&\nonumber	||E_{k,M}||_2^2=\sum^{\infty}_{n=2^k}\sum^{\infty}_{m=M}\sum^{\infty}_{s=2^k}\sum^{\infty}_{r=M}\int_{0}^{1}c_{nm}c_{sr} P^2\psi_{nm}(t)P^2\psi_{sr}(t)dt,\\
	\label{P3_con9}
	&&||E_{k,M}||_2^2\leq \sum^{\infty}_{n=2^k}\sum^{\infty}_{m=M}\sum^{\infty}_{s=2^k}\sum^{\infty}_{r=M}\int_{0}^{1}|c_{nm}||c_{sr}| |P^2\psi_{nm}(t)||P^2\psi_{sr}(t)|dt.
	\end{eqnarray}
	Now, as $t  \in [0,1]$
	\begin{eqnarray*}\label{P3_con7}
		|P^{2}\psi_{nm}(t)|&\leq& \int_{0}^{t}\int_{0}^{t}|\psi_{nm}(t)|dtdt,\\
	\label{P3_con10}	
&\leq& \int_{0}^{t}\int_{0}^{1}|\psi_{nm}(t)|dtdt. 
	\end{eqnarray*}
	Now by \eqref{Pu}, we have
	\begin{eqnarray*}\label{P3_con11}
		|P^{2}\psi_{nm}(t)|\leq  2^{k/2}v(n)\int_{0}^{t}\int_{\frac{\hat{n}-1}{2^k}}^{\frac{\hat{n}+1}{2^k}}|O_{m}(2^{k}t-\hat{n})|dtdt.
	\end{eqnarray*}
	By changing variable $2^kt-\hat{n}=y$ ,we get
	\begin{eqnarray*}
		|P^{2}\psi_{nm}(t)|\leq 2^{-k/2}v(n)\int_{0}^{t}\int_{-1}^{1}|O_{m}(y)|dydt.
	\end{eqnarray*}
	Since $|O_{m}(y)| \leq K_{m}$, $\int_{-1}^{1}|O_{m}(y)| \leq 2 K_{m}$, hence
	\begin{eqnarray*}\label{P3_con12}
		|P^{2}\psi_{nm}(t)|\leq 2^{-k/2}v(n)\int_{0}^{t}2 K_{m} dt.
	\end{eqnarray*}
	Since $t \in [0,1]$, we arrive at the bound
	\begin{eqnarray}\label{P3_con13}
	|P^{2}\psi_{nm}(t)|\leq 2^{-k/2+1}v(n)K_{m}.
	\end{eqnarray}
	Since
	\begin{eqnarray}\label{P3_con14}
	c_{nm}=\int_{0}^{1}f(t)\psi_{nm}(t)w_m(t)dt,
	\end{eqnarray}
	we have
	\begin{eqnarray}\label{P3_con15}
		|c_{nm}| \leq \int_{0}^{1}|f(t)||\psi_{nm}(t)||w_m(t)|dt.
	\end{eqnarray}
	Now using \eqref{P3_con3}, we get
	\begin{eqnarray*}\label{P3_con20}.
		|c_{nm}| \leq \eta\int_{0}^{1}|\psi_{nm}(t)||w_m(t)|dt,
	\end{eqnarray*}
	and by \eqref{Pu}, we have
	\begin{eqnarray*}\label{P3_con16}
		|c_{nm}|\leq  2^{k/2}\eta v(n)\int_{\frac{\hat{n}-1}{2^k}}^{\frac{\hat{n}+1}{2^k}}|O_{m}(2^{k}t-\hat{n})||w_m(2^{k}t-\hat{n})|dt.
	\end{eqnarray*}
	Now by change of variable $2^kt-\hat{n}=y$, we get
	\begin{eqnarray*}\label{P3_con17}
		|c_{nm}|\leq  2^{-k/2}\eta v(n)\int_{-1}^{1}|O_{m}(y)||w_m(y)|dy.
	\end{eqnarray*}
Putting the particular values of $O_m, w_m$ and $v_n$ for different orthogonal polynomial wavelets, we can always solve these inequalities for convergence. For Hermite wavelet we proceed other cases shall follow similarly
\begin{eqnarray*}\label{P2_con17}
	&&|c_{nm}|\leq  2^{-k/2}\frac{\eta}{\sqrt{n!2^{n}\sqrt{\pi}}}\int_{-1}^{1}|H_{m}y|dy,\\
	&&\label{P2_con18} |c_{nm}|\leq  2^{-k/2}\frac{\eta}{\sqrt{n!2^{n}\sqrt{\pi}}}\int_{-1}^{1}\left|\frac{H'_{m+1}(y)}{m+1}\right|dy.
\end{eqnarray*}
By putting $\int_{-1}^{1}|H'_{m+1}(y)|dy = h$, we have
\begin{eqnarray}\label{P2_con19}
|c_{nm}|\leq 2^{-k/2}\frac{1}{(\sqrt{n!2^{n}\sqrt{\pi}})(m+1)}\eta h.
\end{eqnarray}
Putting these values in \eqref{P3_con9}, we have, 
\begin{eqnarray}\label{P2_con21}
&&\|E_{k,M}\|_2^2\leq 2^{-2k}\eta^2h^4\sum^{\infty}_{n=2^k}\sum^{\infty}_{m=M}\sum^{\infty}_{s=2^k}\sum^{\infty}_{r=M}\int_{0}^{1}\frac{1}{(\sqrt{n!2^{n}\sqrt{\pi}})^2(m+1)^2}\frac{1}{(\sqrt{s!2^{s}\sqrt{\pi}})^2(r+1)^2} dt,\\
&&\label{P2_con22}
\|E_{k,M}\|_2^2\leq 2^{-2k}\eta^2h^4\sum^{\infty}_{n=2^k}\frac{1}{n!2^{n}\sqrt{\pi}}\sum^{\infty}_{s=2^k}\frac{1}{s!2^{s}\sqrt{\pi}} \sum^{\infty}_{m=M}\frac{1}{(m+1)^2}\sum^{\infty}_{r=M}\frac{1}{(r+1)^2}.
\end{eqnarray}
Since all four series converge, we have $\|E_{k,M}\| \longrightarrow 0$ as $k,M \rightarrow \infty$.
\end{proof}
\section{Numerical illustrations}\label{P3_sec9}
In this section we apply ChWNA, GeWNA, LeWNA, LaWNA, HeWNA, ChWQA, GeWQA, LeWQA, LaWQA and HeWQA methods to  solve the test examples from real life and compare our solutions with exact solutions whenever available. 

Since we use the Newton Raphson method or Newton's quazilinearization coupled with different wavelet methods, the order of convergence of the proposed method may be at most 2. By using the following error estimates, we shall be analysing the order of convergence for different wavelet methods.

We define $L_\infty$ error, $L_2$ error and consider rate of convergence from \cite{MAJAK2015321} as follows 
\begin{eqnarray}
\label{LInfinity}L_\infty~\mbox{error}&=&\max_{j}|y(t_j)-y_{w,N}(t_j)|,\\
\label{L2}L_2~\mbox{error}&=&\left(\sum_{j=0}^{M-1}|y(t_j)-y_{w,N}(t_j)|^2\right)^{1/2},\\
\mbox{Rate~ of~ Convergence~ (ROC)}&=& \frac{1}{\log{2}}\log\left|\frac{y(t)-y_{w,N}(t)}{y(t)-y_{w,2N}(t)}\right|.
\end{eqnarray}
where $y(t_j)$ is the exact solution and $y_{w,N}(t_j)$ is the wavelets solution at the point $t_j$ with $N$ grid points. Now we consider the some test examples and verify the applicability and accuracy of the proposed wavelets methods.
\subsection{Example 1}\label{P3_sub10}
Consider the non-linear SBVP:
\begin{eqnarray}\label{P3_40}
y''(t)+\frac{2}{t}y'(t)+y^{5}(t)=0,\quad\quad y'(0)=0,\quad y(1)=\sqrt{\frac{3}{4}}.
\end{eqnarray}
Chandrasekhar (\cite{CS1967}, p88) has derived above two point nonlinear SBVP. This equation arise in the study of stellar structure. Its exact solution is $y(t)=\sqrt{\frac{3}{3+t^2}}$. Solutions, errors and ROC are tabulated in tables \ref{P3_ex1tab1}, \ref{P3_ex1tab2}, \ref{P3_ex1tab3} and \ref{P3_ex1tab4}. Solution are also plotted in \ref{ex1fig}  We also observed for small changes in initial vector, (e.g., taking $[0.8,0.8,\hdots,0.8]$ or $[0.7,0.7,\hdots,0.7]$) does not significantly change the solution. Since all these methods converge to the same solution for computation of error and ROC, we have considered solution by only one of the methods. 
\begin{table}[H]
\caption{\small{Comparison of the solutions computed by the proposed methods for example \ref{P3_sub10} at $J=2$ :}}\label{P3_ex1tab1}											 
\centering											
\begin{center}											
\resizebox{16.8cm}{2.01cm}{											
\begin{tabular}	{|c | l|  l|  l| l| l| l| l| l| l| l| l| l|}\hline										
Grid Points	&	ChWNA	&	GeWNA	&	HeWNA	&	LaWNA	&	LeWNA	&	ChWQA	&	GeWQA	&	HeWQA	&	LaWNA	&	LeWQA &Exact	\\\hline
0	&	0.999999992	&	0.999999992	&	0.999999992	&	0.999999992	&	0.999999992	&	0.999999992	&	0.9999999923	&	0.999999992	&	0.999999992	&	0.999999992	& 1\\
0.1	&	0.998337474	&	0.998337474	&	0.998337474	&	0.998337474	&	0.998337474	&	0.998337474	&	0.998337474	&	0.998337474	&	0.998337474	&	0.998337474 &	0.998337488	\\
0.2	&	0.993399259	&	0.993399259	&	0.993399259	&	0.993399259	&	0.993399259	&	0.993399259	&	0.993399259	&	0.993399259	&	0.993399259	&	0.993399259 & 0.993399268	\\
0.3	&	0.985329271	&	0.985329271	&	0.985329271	&	0.985329271	&	0.985329271	&	0.985329271	&	0.985329271	&	0.985329271	&	0.985329271	&	0.985329271 &	0.985329278	\\
0.4	&	0.974354698	&	0.974354698	&	0.974354698	&	0.974354698	&	0.974354698	&	0.974354698	&	0.974354698	&	0.974354698	&	0.974354698	&	0.974354698 & 0.974354704	\\
0.5	&	0.960768918	&	0.960768918	&	0.960768918	&	0.960768918	&	0.960768918	&	0.960768918	&	0.960768918	&	0.960768918	&	0.960768918	&	0.960768918  &	0.960768923	\\
0.6	&	0.944911178	&	0.944911178	&	0.944911178	&	0.944911178	&	0.944911178	&	0.944911178	&	0.944911178	&	0.944911178	&	0.944911178	&	0.944911178 &	0.944911183	\\
0.7	&	0.927145538	&	0.927145538	&	0.927145538	&	0.927145538	&	0.927145538	&	0.927145538	&	0.927145538	&	0.927145538	&	0.927145538	&	0.927145538  & 0.927145541	\\
0.8	&	0.907841296	&	0.907841296	&	0.907841296	&	0.907841296	&	0.907841296	&	0.907841296	&	0.907841296	&	0.907841296	&	0.907841296	&	0.907841296 &	0.907841299	\\
0.9	&	0.887356507	&	0.887356507	&	0.887356507	&	0.887356507	&	0.887356507	&	0.887356507	&	0.887356507	&	0.887356507	&	0.887356507	&	0.887356507	 & 0.887356509 \\
1	&	0.866025404	&	0.866025404	&	0.866025404	&	0.866025404	&	0.866025404	&	0.866025404	&	0.866025404	&	0.866025404	&	0.866025404	&	0.866025404 &	0.866025404 \\\hline
\end{tabular}}											
\end{center}	
\end{table}

\begin{table}[H]
	\caption{\small{Errors for example \ref{P3_sub10} at $J=2$ :}}\label{P3_ex1tab2}											 										
	\centering											
	\begin{center}											
		\resizebox{4cm}{0.6cm}{											
			\begin{tabular}	{|c |c |c |}										
				\hline											
				
				Error	&	  WNA	&	WQA	\\\hline
				$L_\infty$	& 1.43E-08 &  1.43E-08\\
				$L_2$	& 2.07069E-08 & 2.07069E-08\\\hline
		\end{tabular}}											
	\end{center}
\end{table}

\begin{table}[H]
	\caption{\small{Absolute errors and ROC for wavelets Newton approach at $t=0.5$ for example \ref{P3_sub10}:}}\label{P3_ex1tab3}
\centering											
\begin{center}											
	\resizebox{8cm}{1cm}{											
		\begin{tabular}{|cccc|}									
			\hline
			Grid Points	&	Solution	&	Absolute Error	&	ROC\\\hline
			2	&	0.959325955	&	0.001442967	&		\\
			4	&	0.960766904	&	2.02E-06	&	9.481227721	\\
			8	&	0.960768918	&	4.70E-09	&	8.746543106	\\
			16	&	0.960768923	&	1.60E-14	&	18.16555862\\\hline
		\end{tabular}}											
	\end{center}
\end{table}

\begin{table}[H]
	\caption{\small{Absolute errors and ROC for wavelets quasilinearization approach at $t=0.5$ for example \ref{P3_sub10}:}}\label{P3_ex1tab4}	\centering											
\begin{center}											
	\resizebox{8cm}{1cm}{											
		\begin{tabular}{|cccc|}									
			\hline	
				Grid Points	&	Solution	&	Absolute Error	&	ROC\\\hline
			2	&	0.959325955	&	0.001442967	&		\\
			4	&	0.960766904	&	2.01894E-06	&	9.481227721	\\
			8	&	0.960768918	&	4.70057E-09	&	8.74654314	\\
			16	&	0.960768923	&	1.59872E-14	&	18.16555862 \\\hline
		    \end{tabular}}											
	\end{center}
\end{table}

\begin{figure}[H]
\begin{center}
\includegraphics[scale=0.30]{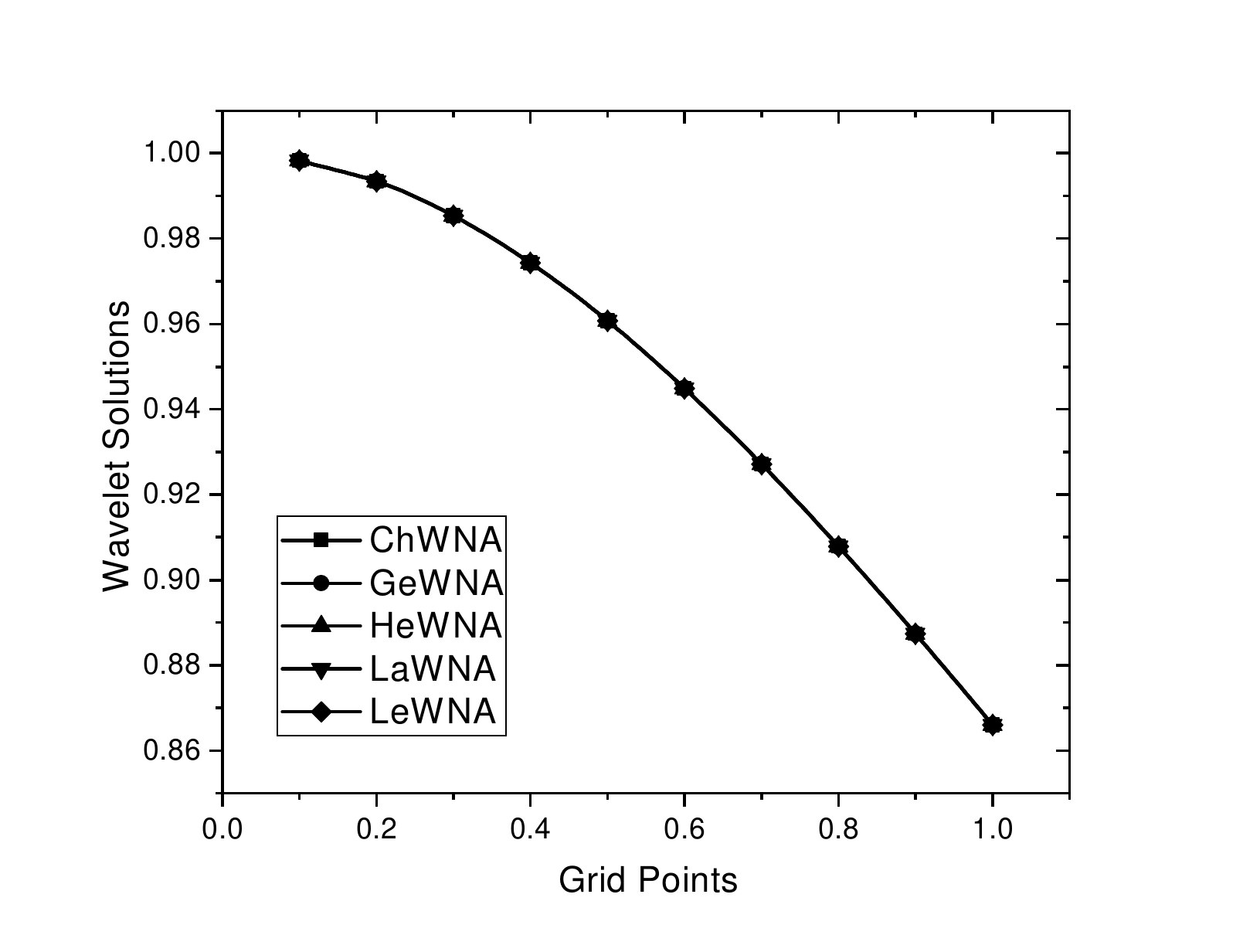}\includegraphics[scale=0.30]{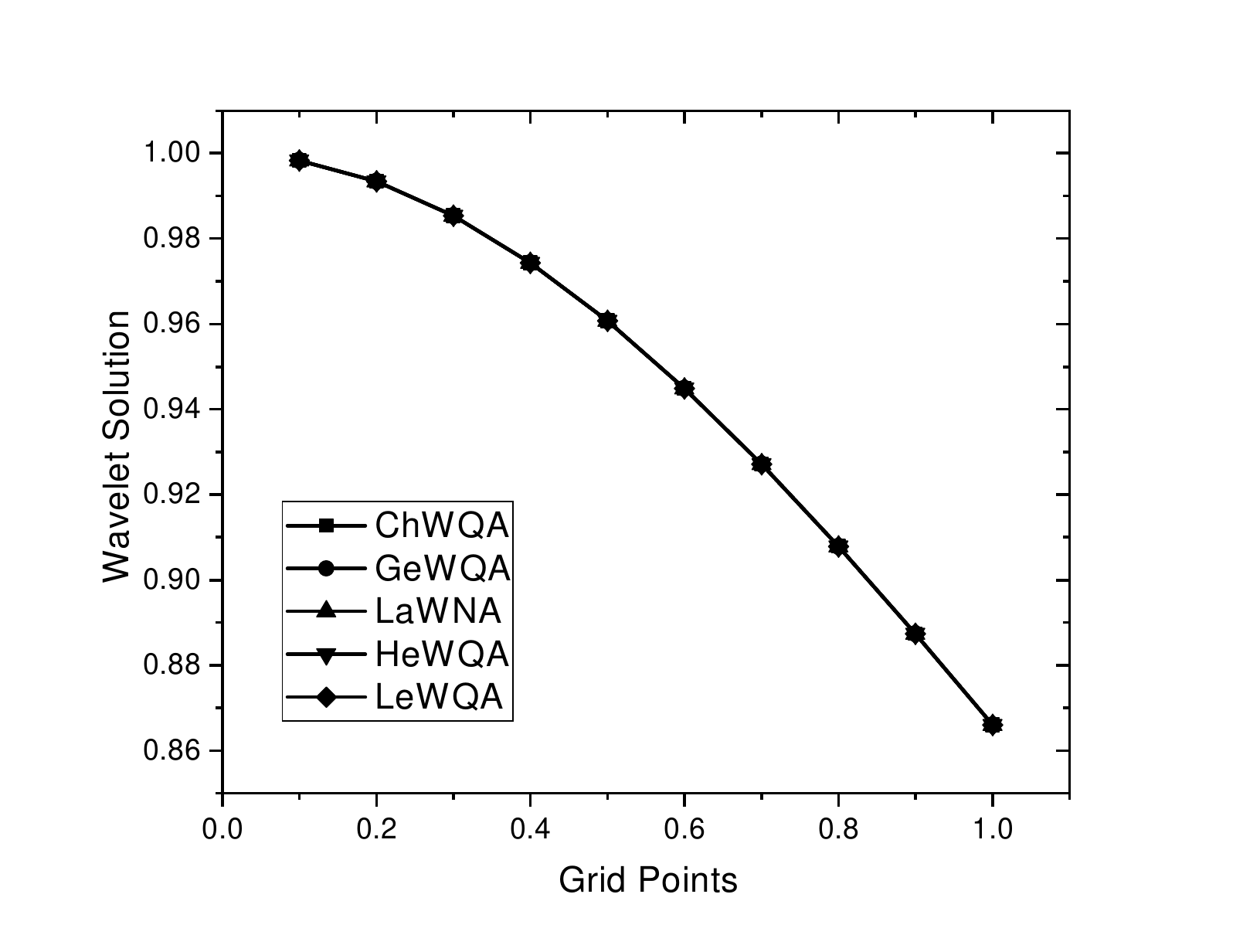}
\end{center}
\caption{Solution plots for $J=3$ for example \ref{P3_sub10} for Newton and Quasilinearization approach}\label{ex1fig}
\end{figure}

\subsection{Example 2}\label{p3_sub11}
Consider the nonlinear SBVP:
\begin{eqnarray}\label{P3_41}
y''(t)+\frac{1}{t}y'(t)+e^{y(t)}=0,\quad\quad y'(0)=0,\quad y(1)=0.
\end{eqnarray}
It is derived by Chamber \cite{CHAMBRE1952} to study the thermal explosion in a cylindrical vessel. The exact solution of \eqref{P3_41} is $$y(t) = 2 \ln{\left(\frac{4-2\sqrt{2}}{(3-2\sqrt{2})t^2+1}\right)}.$$ Solutions, errors and ROC are tabulated in tables \ref{P3_ex2tab1}, \ref{P3_ex2tab2}, \ref{P3_ex2tab3} and \ref{P3_ex2tab4}. Solution are also plotted in \ref{ex2fig}. We also observed that for small changes in the initial vector, (e.g., $[0.1,0.1,\hdots,0.1]$ or $[0.2,0.2,\hdots,0.2]$) does not significantly change the solution. Since all these methods converge to the same solution for computation of error and ROC, we have considered solution by only one of the methods.  
\begin{table}[H]
\caption{\small{Comparison of the solutions computed by the proposed methods for example \ref{p3_sub11} at $J=2$:}}\label{P3_ex2tab1}			\centering											 
\begin{center}											
\resizebox{16.8cm}{2.01cm}{											
\begin{tabular}	{|c | l|  l|  l| l| l| l| l| l| l| l| l| l|}										
\hline											
Grid Point	&	ChWNA	&	GeWNA	&	HeWNA	&	LaWNA	&	LeWNA	&	ChWQA	&	GeWQA	&	HeWQA	&	LaWQA	&	LeWQA & Exact	\\\hline
0	&	0.316694368	&	0.316694368	&	0.316694368	&	0.316694368	&	0.316694368	&	0.316694368	&	0.316694368	&	0.316694368	&	0.316694368	&	0.316694368	& 0.316694368 \\
0.1	&	0.31326585	&	0.31326585	&	0.31326585	&	0.31326585	&	0.31326585	&	0.31326585	&	0.31326585	&	0.31326585	&	0.31326585	&	0.31326585 & 0.31326585	\\
0.2	&	0.303015423	&	0.303015423	&	0.303015423	&	0.303015423	&	0.303015423	&	0.303015423	&	0.303015423	&	0.303015423	&	0.303015423	&	0.303015423 & 0.303015423 \\
0.3	&	0.286047265	&	0.286047265	&	0.286047265	&	0.286047265	&	0.286047265	&	0.286047265	&	0.286047265	&	0.286047265	&	0.286047265	&	0.286047265 & 0.286047265 \\
0.4	&	0.262531127	&	0.262531127	&	0.262531127	&	0.262531127	&	0.262531127	&	0.262531127	&	0.262531127	&	0.262531127	&	0.262531127	&	0.262531127 & 0.262531127 \\
0.5	&	0.232696784	&	0.232696784	&	0.232696784	&	0.232696784	&	0.232696784	&	0.232696784	&	0.232696784	&	0.232696784	&	0.232696784	&	0.232696784 & 0.232696784 \\
0.6	&	0.196826806	&	0.196826806	&	0.196826806	&	0.196826806	&	0.196826806	&	0.196826806	&	0.196826806	&	0.196826806	&	0.196826806	&	0.196826806 & 0.196826806 \\
0.7	&	0.155248107	&	0.155248107	&	0.155248107	&	0.155248107	&	0.155248107	&	0.155248107	&	0.155248107	&	0.155248107	&	0.155248107	&	0.155248107 & 0.155248107 \\
0.8	&	0.108322763	&	0.108322763	&	0.108322763	&	0.108322763	&	0.108322763	&	0.108322763	&	0.108322763	&	0.108322763	&	0.108322763	&	0.108322763	& 0.108322763 \\
0.9	&	0.056438602	&	0.056438602	&	0.056438602	&	0.056438602	&	0.056438602	&	0.056438602	&	0.056438602	&	0.056438602	&	0.056438602	&	0.056438602 & 0.056438602 \\
1	&	0	&	0	&	0	&	0	&	0	&	0	&	0	&	0	&	0	&	0 & 0	\\\hline
\end{tabular}}											
\end{center}											
\label{Table2}											
\end{table}

\begin{table}[H]
	\caption{\small{Errors for example \ref{p3_sub11} at $J=2$ :}}\label{P3_ex2tab2}											 \centering											
	\centering											
	\begin{center}											
		\resizebox{4cm}{0.6cm}{											
			\begin{tabular}	{|c |c| c |  }										
				\hline											
				
				Error	&	  WNA	&	WQA	\\\hline
				$L_\infty$	& 3.00502E-10
				 & 3.00502E-10
				 \\
				$L_2$	&5.46469E-10 & 5.46469E-10\\\hline
		\end{tabular}}											
	\end{center}
\end{table}

\begin{table}[H]
	\caption{\small{Absolute errors and ROC for wavelets Newton approach at $t=0.5$ for example \ref{p3_sub11}:}}\label{P3_ex2tab3}												 									
	\centering											
	\begin{center}											
		\resizebox{8cm}{1cm}{											
		\begin{tabular}{|cccc|}
				\hline											
				
				Grid Points	&	Solution	&	Absolute Error	&	ROC\\\hline
			2	&	0.231385398	&	0.001311386	&		\\
			4	&	0.232698683	&	1.90E-06	&	9.431803614	\\
			8	&	0.232696784	&	1.52E-10	&	13.61147493	\\
			16	&	0.232696784 &	4.72E-16	&	18.2945666413\\\hline	
			\end{tabular}}											
	\end{center}
\end{table}

\begin{table}[H]
	\caption{\small{Absolute errors and ROC for wavelet quasilinearization approach at $t=0.5$ for example \ref{p3_sub11}:}}\label{P3_ex2tab4}												 									
	\centering											
	\begin{center}											
\resizebox{8cm}{1cm}{											
		\begin{tabular}{|cccc|}				\hline											
				
				Grid Points	&	Solution	&	Absolute Error	&	ROC\\\hline
			2	&	0.231385398	&	0.001311386	&		\\
			4	&	0.232698683	&	1.90E-06	&	9.431803614	\\
			8	&	0.232696784	&	1.52E-10	&	13.61147334	\\
			16	&	0.232696784&	4.72E-16	&	18.2945666413\\\hline	
				\end{tabular}}											
	\end{center}
\end{table}

\begin{figure}[H]
\begin{center}
\includegraphics[scale=0.30]{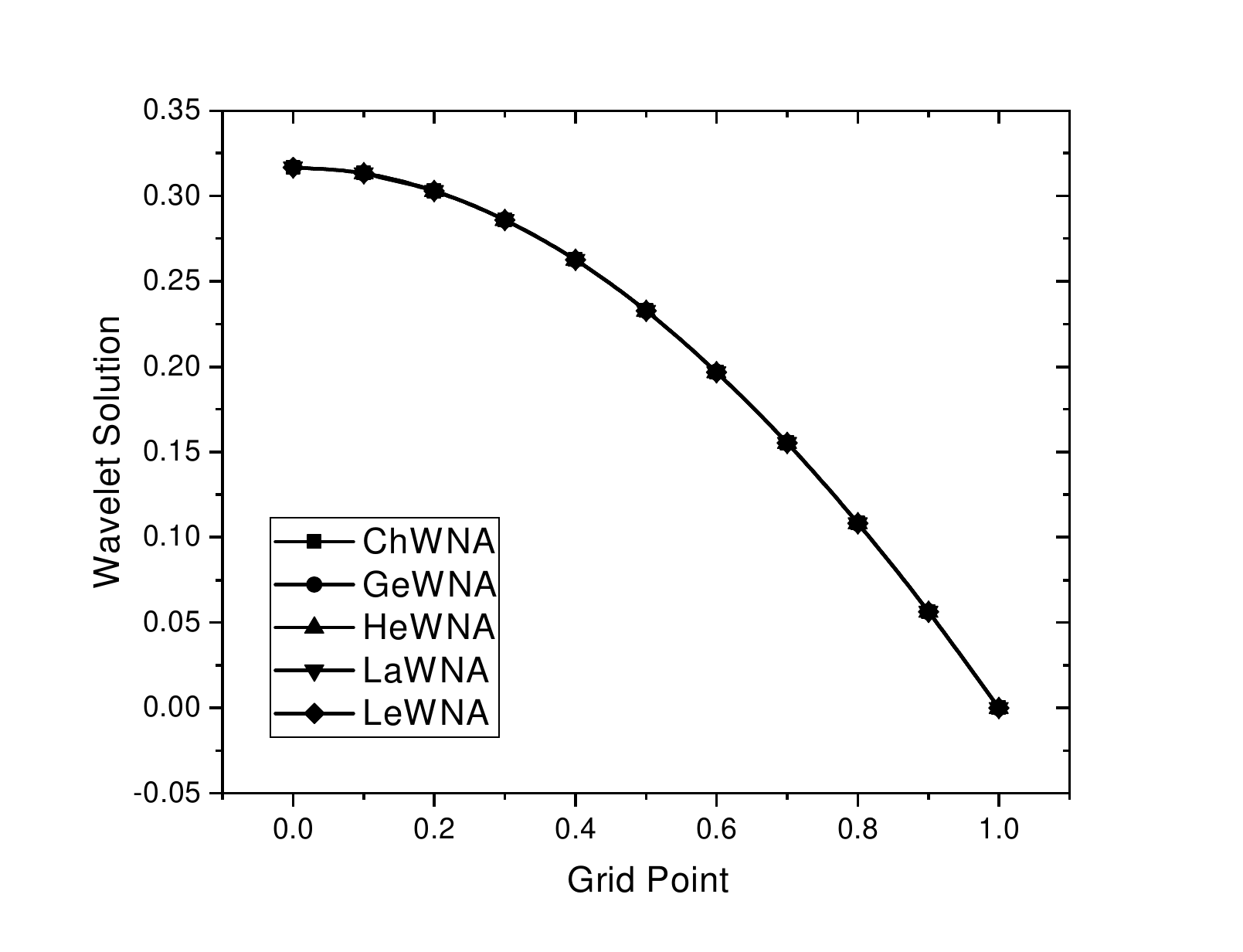}\includegraphics[scale=0.30]{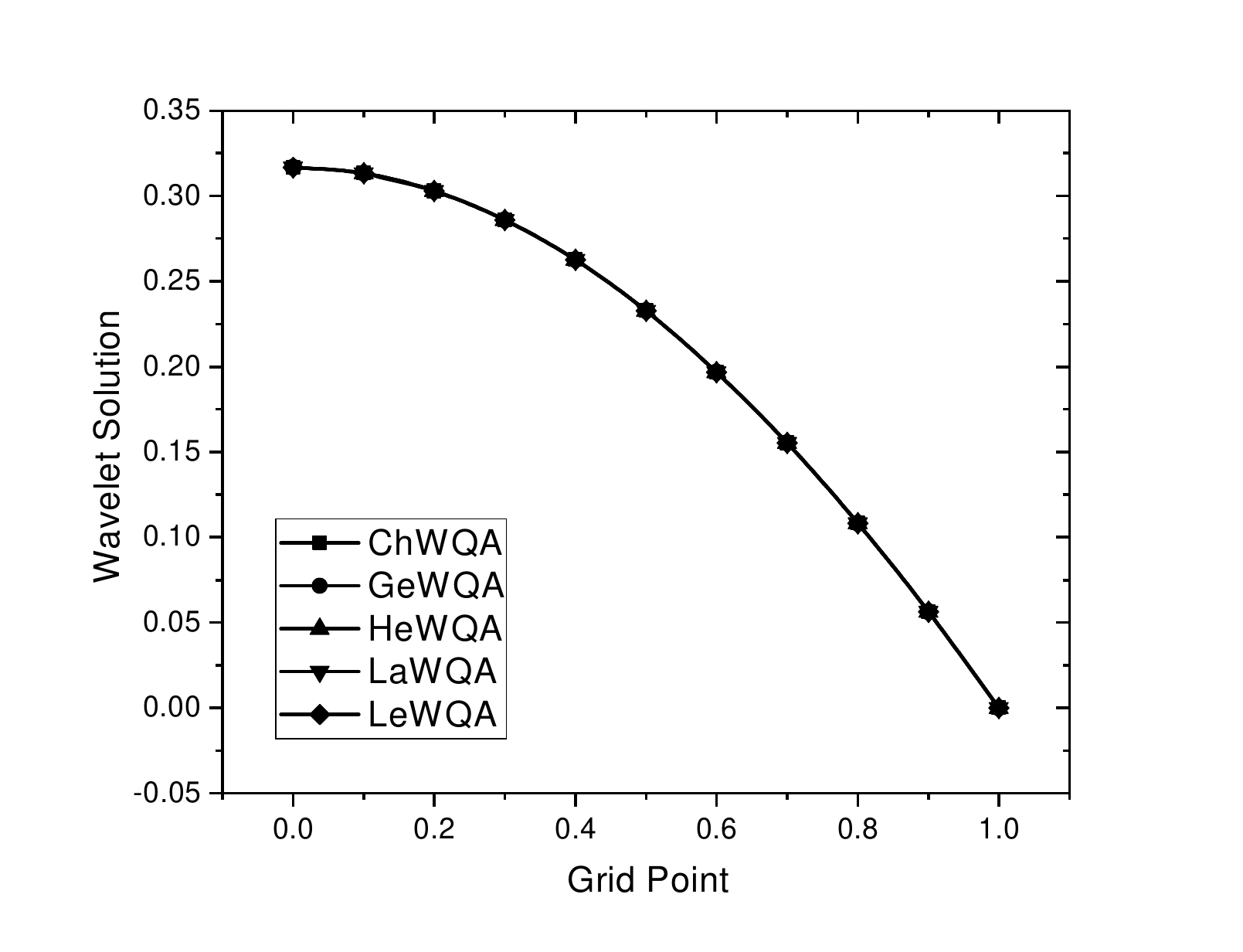}
\end{center}
\caption{Solution plots for $J=3$ for example \ref{p3_sub11} for Newton and Quasilinearization approach}\label{ex2fig}
\end{figure}

\subsection{Example 3}\label{p3_sub12}
Consider the nonlinear SBVP:
\begin{eqnarray}\label{P3_42}
y''(t)+\frac{3}{t}y'(t)+\left(\frac{1}{8y^{2}}-\frac{1}{2}\right)=0,\quad\quad y'(0)=0,\quad y(1)=1.
\end{eqnarray}
The above nonlinear SBVP is discussed in \cite{RW1989,JV1998} to study rotationally symmetric solutions of shallow membrane caps. Exact solution of this problem is not known. Computed solutions are tabulated in table \ref{P3_extab3}. Solution are also plotted in \ref{ex3fig}. We have given table \ref{P3VIM1} to illustrate the accuracy of the computed solution in absence of exact solutions. We also observed that for small changes in initial vector, (e.g., taking $[0.9,0.9,\hdots,0.9]$ or $[0.8,0.8,\hdots,0.8]$) does not significantly change the solution. 
\begin{table}[H]
\caption{\small{Numerical solution of example \ref{p3_sub12} computed in \cite{Mandeep-Amit-2016} by VIM+HPM:}}\label{P3VIM1}														 \centering	
\begin{center}	
\resizebox{12cm}{0.45cm}{											
\begin{tabular}	{|cccccccccccc|}											
\hline	
$t$ &0 &0.1&0.2&0.3&0.4&0.5&0.6&0.7&0.8&0.9&1\\\hline
\mbox{VIM+HPM}&0.954135&0.954589&0.95595&0.958822&0.961403&0.965503&0.970526&0.976479&0.983369&0.991206&1
\\\hline
\end{tabular}}											
\end{center}
\end{table}
\begin{table}[H]
\caption{\small{Comparison of the solutions computed by the proposed methods for example \ref{p3_sub12} at $J=2$:}}\label{P3_extab3}														 \centering	
\begin{center}	
\resizebox{16.8cm}{2.01cm}{											
\begin{tabular}	{|c | l|  l|  l| l| l| l| l| l| l| l|}											
\hline	
Grid Point	&	ChWNA	&	GeWNA	&	HeWNA	&	LaWNA	&	LeWNA	&	ChWQA	&	GeWQA	&	HeWQA	&	LaWQA	&	LeWQA	\\\hline
0	&	0.954135307	&	0.954135307	&	0.954135307	&	0.954135307	&	0.954135307	&	0.954135328	&	0.954135328	&	0.954135307	&	0.954135307	&	0.954135307	\\
0.1	&	0.954588729	&	0.954588729	&	0.954588729	&	0.954588729	&	0.954588729	&	0.954588759	&	0.954588759	&	0.954588759	&	0.954588729	&	0.954588729	\\
0.2	&	0.955949645	&	0.955949645	&	0.955949645	&	0.955949645	&	0.955949645	&	0.955949678	&	0.955949678	&	0.955949678	&	0.955949645	&	0.955949645	\\
0.3	&	0.958220005	&	0.958220005	&	0.958220005	&	0.958220005	&	0.958220005	&	0.958220025	&	0.958220025	&	0.958220025	&	0.958220005	&	0.958220005	\\
0.4	&	0.961403036	&	0.961403036	&	0.961403036	&	0.961403036	&	0.961403036	&	0.961403044	&	0.961403044	&	0.961403044	&	0.961403036	&	0.961403036	\\
0.5	&	0.965503219	&	0.965503219	&	0.965503219	&	0.965503219	&	0.965503219	&	0.965503224	&	0.965503224	&	0.965503224	&	0.965503219	&	0.965503219	\\
0.6	&	0.970526246	&	0.970526246	&	0.970526246	&	0.970526246	&	0.970526246	&	0.970526259	&	0.970526259	&	0.970526259	&	0.970526246	&	0.970526246	\\
0.7	&	0.97647897	&	0.97647897	&	0.97647897	&	0.97647897	&	0.97647897	&	0.97647899	&	0.97647899	&	0.97647899	&	0.97647897	&	0.97647897	\\
0.8	&	0.983369349	&	0.983369349	&	0.983369349	&	0.983369349	&	0.983369349	&	0.983369362	&	0.983369362	&	0.983369362	&	0.983369349	&	0.983369349	\\
0.9	&	0.991206375	&	0.991206375	&	0.991206375	&	0.991206375	&	0.991206375	&	0.991206367	&	0.991206367	&	0.991206367	&	0.991206375	&	0.991206375	\\
1	&	1	&	1	&	1	&	1	&	1	&	1	&	1	&	1	&	1	&	1	\\\hline
\end{tabular}}											
\end{center}
\end{table}

\begin{figure}[H]
\begin{center}
\includegraphics[scale=0.30]{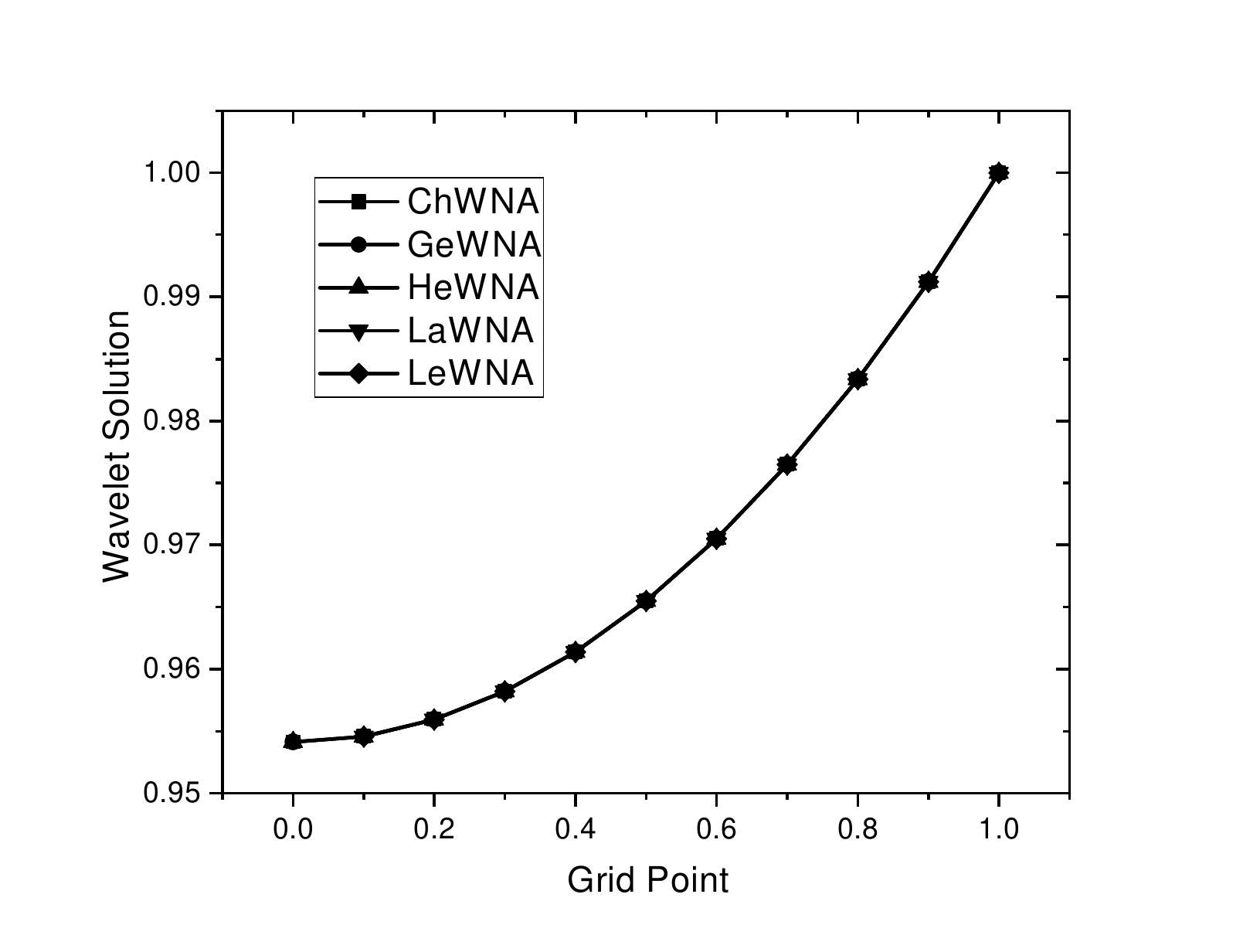}\includegraphics[scale=0.30]{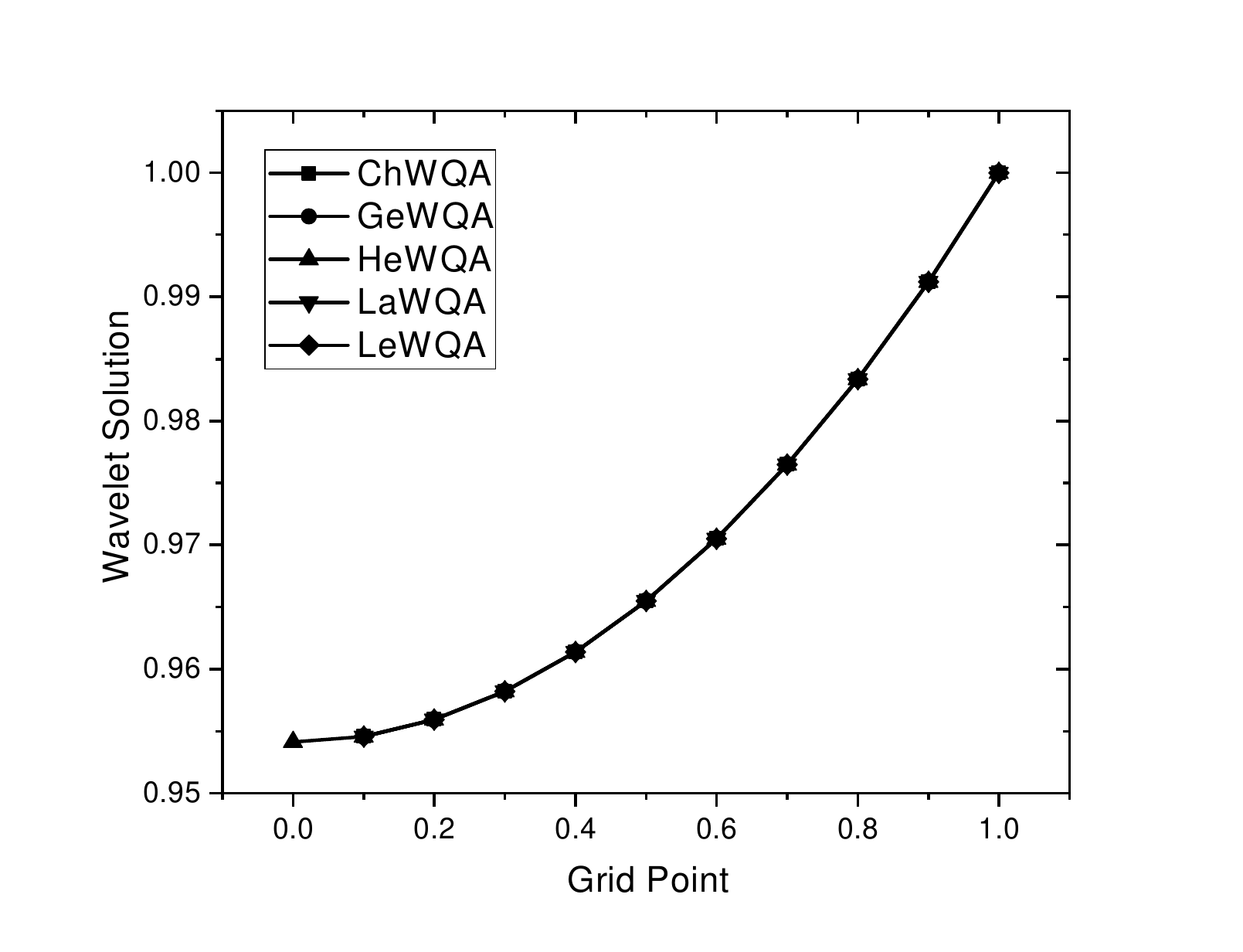}
\end{center}
\caption{Solution plots for $J=3$ for example \ref{p3_sub12} for Newton and Quasilinearization approach}\label{ex3fig}
\end{figure}

\subsection{Example 4}\label{p3_sub13}
Consider the nonlinear SBVP:
\begin{eqnarray}\label{P3_83}
y''(t)+\frac{2}{t}y'(t)+e^{-y(t)}=0,\quad\quad y'(0)=0,\quad 2y(1)+y'(1)=0.
\end{eqnarray}
The above nonlinear SBVP is discussed by Duggan and Goodman \cite{RA1986} as heat conduction model in human head. Exact solution of this problem is not known to the best of our knowledge. Computed solutions are tabulated in table \ref{P3_extab4}.  Solution are also plotted in \ref{ex4fig}. We have given table \ref{P3VIM2} to illustrate the accuracy of the computed solution in absence of exact solutions. We also observed that for small changes in initial vector, (e.g., taking $[0.1,0.1,\hdots,0.1]$ or $[0.2,0.2,\hdots,0.2]$) does not significantly change the solution. 
\begin{table}[H]
\caption{\small{Numerical solution of example \ref{p3_sub13} computed in \cite{Mandeep-Amit-2016} by VIM+HPM:}}\label{P3VIM2}														 \centering	
\begin{center}	
\resizebox{12cm}{0.45cm}{											
\begin{tabular}	{|cccccccccccc|}											
\hline	
$t$ &0 &0.1&0.2&0.3&0.4&0.5&0.6&0.7&0.8&0.9&1\\\hline
\mbox{VIM+HPM}& 0.269896 & 0.268624 & 0.264804 & 0.258416 & 0.249432 & 0.237809 & 0.223491 & 0.206408 & 0.186477 & 0.163596 & 0.137646
\\\hline
\end{tabular}}											
\end{center}
\end{table}

\begin{table}[H]
\caption{\small{ Comparison of the solutions computed by the proposed methods for example \ref{p3_sub13} at $J=2$:}}\label{P3_extab4}														 \centering											
\begin{center}	
\resizebox{16.8cm}{2.01cm}{											
\begin{tabular}	{|c | l|  l|  l| l| l| l| l| l| l| l|}										
\hline		
Grid Point	&	ChWNA	&	GeWNA	&	HeWNA	&	LaWNA	&		LeWNA	&	ChWQA	&	GeWQA	&	HeWQA	&	LaWQA	&	LeWQA	\\\hline
0	&	0.269943779	&	0.269943779	&	0.269943779	&	0.269948773	&		0.269943779	&	0.272366649	&	0.272359192	&	0.272340537	&	0.272340537	&	0.272360434	\\
0.1	&	0.268676386	&	0.268676386	&	0.268676386	&	0.268676385	&		0.268676386	&	0.271096835	&	0.271089425	&	0.271070957	&	0.271070957	&	0.27109066	\\
0.2	&	0.264853383	&	0.264853383	&	0.264853383	&	0.264853383	&		0.264853383	&	0.267280011	&	0.267272742	&	0.267254819	&	0.267254819	&	0.267273954	\\
0.3	&	0.258462168	&	0.258462168	&	0.258462168	&	0.258462168	&		0.258462168	&	0.260894018	&	0.260886981	&	0.260869927	&	0.260869927	&	0.260888154	\\
0.4	&	0.24947312	&	0.24947312	&	0.24947312	&	0.24947312	&		0.24947312	&	0.251901887	&	0.251895169	&	0.251879256	&	0.251879256	&	0.251896289	\\
0.5	&	0.237844161	&	0.237844161	&	0.237844161	&	0.237844161	&		0.237844161	&	0.240251793	&	0.240245476	&	0.24023091	&	0.24023091	&	0.240246529	\\
0.6	&	0.223520119	&	0.223520119	&	0.223520119	&	0.223520119	&		0.223520119	&	0.225877002	&	0.225871161	&	0.225858081	&	0.225858081	&	0.225872134	\\
0.7	&	0.206431878	&	0.206431878	&	0.206431878	&	0.206431878	&		0.206431878	&	0.208695838	&	0.208690539	&	0.208679016	&	0.208679016	&	0.208691422	\\
0.8	&	0.186495288	&	0.186495288	&	0.186495288	&	0.186495288	&		0.186495288	&	0.188611697	&	0.188606996	&	0.188597035	&	0.188597035	&	0.188607779	\\
0.9	&	0.163609789	&	0.163609789	&	0.163609789	&	0.163609789	&		0.163609789	&	0.16551314	&	0.16550908	&	0.165500639	&	0.165500639	&	0.165509757	\\
1	&	0.137656718	&	0.137656718	&	0.137656718	&	0.137656718	&		0.137656718	&	0.139274129	&	0.139270737	&	0.139263739	&	0.139263739	&	0.139271303	\\\hline
\end{tabular}}											
\end{center}											
\end{table}	

\begin{figure}[H]
\begin{center}
\includegraphics[scale=0.30]{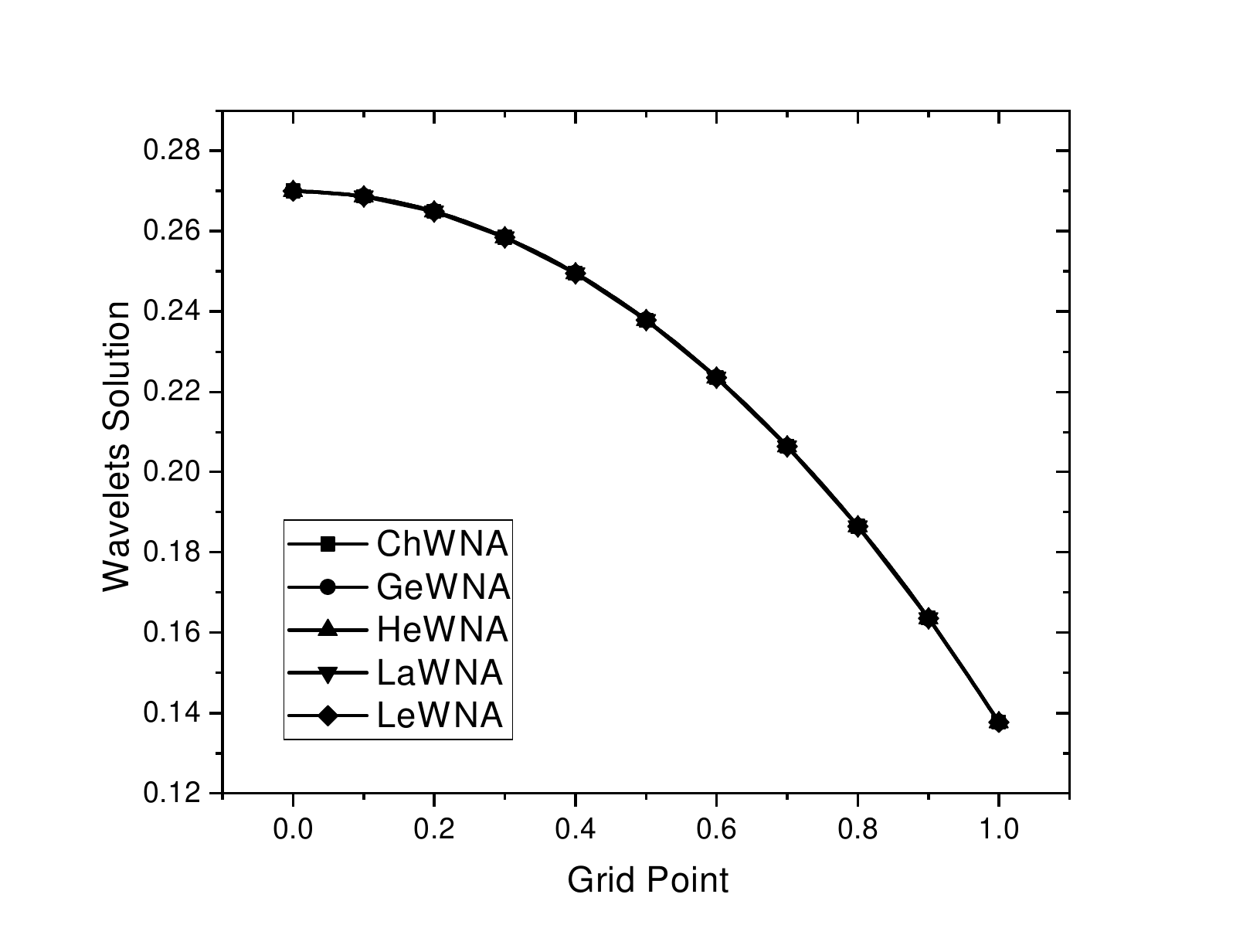}\includegraphics[scale=0.30]{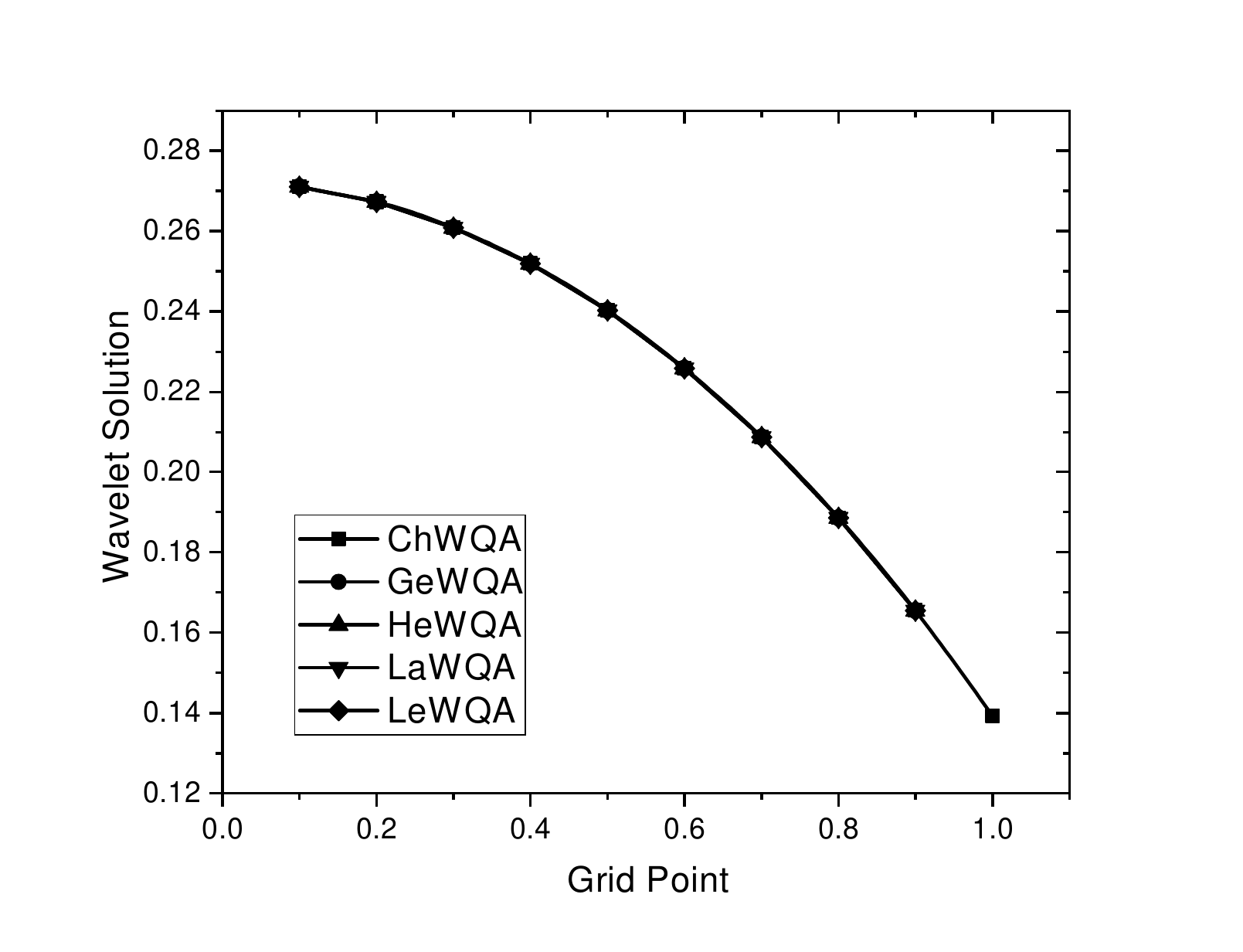}
\end{center}
\caption{Solution plots for $J=3$ for example \ref{p3_sub13} for Newton and Quasilinearization approach}\label{ex4fig}
\end{figure}

\section{Conclusions}\label{P3_conclusions}
 In this research article, we have proposed general methods using orthonormal polynomial wavelets. We construct ten different wavelet methods referred to as ChWNA, GeWNA, LeWNA,  LaWNA, HeWNA, ChWQA, GeWQA, LeWQA, LaWQA and HeWQA. We apply these methods to solve nonlinear SBVPs arising in different branches of science and engineering \cite{RW1989,JV1998, RA1986,CHAMBRE1952,CS1967}. The problem of singularity is handled with the help of this general approach. In methods based on the wavelet Newton approach, we have used the Newton-Raphson method to solve the nonlinear system. In methods based on wavelet quasilinearization, the difficulty arose due to the nonlinearity of differential equations is overcome with the help of quasilinearization approach. The main advantage of proposed methods is that solutions with high accuracy are obtained using few iterations and few spatial divisions. Computational work illustrates the validity and accuracy of the procedure. Figures \ref{ex1fig}, \ref{ex2fig}, \ref{ex3fig} and \ref{ex4fig} we show the nature of the solutions. These figures pictorially verify that the solutions obtained are similar in nature and accuracy, irrespective of the polynomials. We observe the change in initial guesses does not result in large deviations in solutions, i.e., small variations in initial guesses results small changes in solutions so our methods are robust and stable and must be preferred over other existing methods. The theoretical convergence verifies our method very well. The singularities are also dealt in a much easier way.The wavelet methods based on orthonormal polynomial are highly reliable and easy to implement and consume less time. It will be interesting to see how the proposed techniques are explored and extended to other class of problems (\cite{Sunil2014, Rashidi2014, Ranbir2020, Mugisha2020, Bessem2020, Manafian2020, Hossein2020, Fatima2020, Aziz2020}).


\begin{thebibliography}{}

\end{thebibliography}


\begin{thebibliography}{10}
	
	\bibitem{JV1998}
	J. V.~Baxley and S. B.~Robinson,
	\newblock{\em Nonlinear boundary value problems for shallow membrane caps ii},
	\newblock Journal of Computational and Applied Mathematics. 88(1998), 203 --
	224.
	
	\bibitem{BIAZAR2012608}
	J.~Biazar and H.~Ebrahimi,
	\newblock{\em Chebyshev wavelets approach for nonlinear systems of volterra
	integral equations},
	\newblock Computers and Mathematics with Applications. 63(2012), 608 -- 616.
	
	\bibitem{BN2009}
	A.~Boggess and F.~J.~Narcowich,
	\newblock{\em A first course in wavelets with fourier analysis},
	\newblock John Wiley and Sons, 2009.
	
	\bibitem{Carlo2001}
	C. Cattani,
	\newblock{\em Haar wavelet splines},
	\newblock Journal of Interdisciplinary Mathematics. 4(2001), 35--47.
	
	\bibitem{CHAMBRE1952}
	P.~L. Chambre,
	\newblock{\em On the solution of the poisson boltzmann equation with application to
	the theory of thermal explosions},
	\newblock The Journal of Chemical Physics. 20(1952), 1795--1797.
	
	\bibitem{CS1967}
	S.~Chandrasekhar,
	\newblock{\em Introduction to the study of stellar structure},
	\newblock Dover Publications, 1967.
	
	\bibitem{Chawla1988}
	M.~M. Chawla, R.~Subramanian, and H.~L. Sathi,
	\newblock{\em A fourth order method for a singular two-point boundary value
	problem},
	\newblock BIT Numerical Mathematics. 28(1988), 88--97.
	
	\bibitem{CH1997}
	C.~F. Chen and C.~H. Hsiao,
	\newblock{\em Haar wavelet method for solving lumped and distributed-parameter
	systems},
	\newblock IEEE Proceedings Control Theory and Applications. 144(1997), 87 -- 94.
	
	\bibitem{ID1992}
	I. Daubechies,
	\newblock{\em Ten lectures on wavelets},
	\newblock Society For Industrial And Applied Mathematics, 1992.
	
	\bibitem{RW1989}
	R.~W. Dickey,
	\newblock{\em Rotationally symmetric solutions for shallow membrane caps},
	\newblock Quarterly of Applied Mathematics. 47(1989), 571--581.
	
	\bibitem{RA1986}
	R.~C. Duggan and A.~M. Goodman,
	\newblock{\em Pointwise bounds for a nonlinear heat conduction model of the human
	head},
	\newblock Bulletin of Mathematical Biology. 48(1986), 229 -- 236.
	
	\bibitem{Gupta2015}
	A.~K. Gupta and S.~Saha Ray,
	\newblock{\em An investigation with hermite wavelets for accurate solution of
	fractional jaulent-miodek equation associated with energy-dependent schrodinger potential},
	\newblock Applied Mathematics and Computation. 270(2015), 458--471.
	
	\bibitem{IQBAL201550}
	M.~A. Iqbal, U. Saeed, and S.~T. Mohyud-Din,
	\newblock{\em Modified laguerre wavelets method for delay differential equations of
	fractional-order},
	\newblock Egyptian Journal of Basic and Applied Sciences. 2(2015), 50 -- 54.
	
	\bibitem{RC2013}
	H. Kaur, R.C~. Mittal, and V. Mishra,
	\newblock{\em Haar wavelet approximate solutions for the generalized lane-emden
	equations arising in astrophysics},
	\newblock Computer Physics Communications. 184(2013), 2169--2177.
	
	\bibitem{KHELLAT2006}
	F.~Khellat and S.A. Yousefi,
	\newblock{\em The linear legendre mother wavelets operational matrix of integration
	and its application},
	\newblock Journal of the Franklin Institute. 343(2006), 181 -- 190.
	
	\bibitem{Antony2017}
	K. Kumar and V.~A.~Vijesh,
	\newblock{\em Chebyshev wavelet quasilinearization scheme for coupled nonlinear sine-gordon equations},
	\newblock Journal of Computational and Nonlinear Dynamics. 12(2017), 011018.
	
	\bibitem{VH2011}
	H. Maan, R.~C. Mittal, and V. Mishra,
	\newblock{\em Haar wavelet quasilinearization approach for solving nonlinear boundary value problems},
	\newblock American Journal of Computational Mathematics. 1(2011), 176--182.
	
	\bibitem{MAJAK2018CS}
	J.~Majak, M.~Pohlak, K.~Karjust, M.~Eerme, J.~Kurnitski, and B.S. Shvartsman,
	\newblock{\em New higher order haar wavelet method: Application to fgm structures},
	\newblock Composite Structures. 201(2018), 72 -- 78.
	
	\bibitem{MAJAK2019AIP}
	J.~Majak, M. Pohlak, M. Eerme, and B. Shvartsman,
	\newblock{\em Solving ordinary differential equations with higher order haar
	wavelet method},
	\newblock AIP Conference Proceedings. 2116(2019), 330002.
	
	\bibitem{Maleknejad2019}
	K. Maleknejad and A. Hoseingholipour,
	\newblock{\em The impact of legendre wavelet collocation method on the solutions of
	nonlinear system of two-dimensional integral equations},
	\newblock International Journal of Computer Mathematics. 0(2019), 1--16.
	
	\bibitem{VF2001}
	V.~B. Mandelzweig and F.~Tabakin,
	\newblock{\em Quasilinearization approach to nonlinear problems in physics with
	application to nonlinear odes},
	\newblock Computer Physics Communications. 141(2001), 268--281.
	
	\bibitem{Mittal2017}
	R.~C. Mittal and S. Pandit,
	\newblock{\em Sensitivity analysis of shock wave Burgers' equation via a novel algorithm based on scale-3 Haar wavelets},
	\newblock International Journal of Computer Mathematics. 95(2018), 601--625.
	
	\bibitem{MOH2011}
	F.~Mohammadi and M.M. Hosseini,
	\newblock{\em A new legendre wavelet operational matrix of derivative and its
	applications in solving the singular ordinary differential equations},
	\newblock Journal of the Franklin Institute. 348(2011), 1787 -- 1796.
	
	\bibitem{rkpaks2009}
	R. K. Pandey and A.~K. Singh,
	\newblock{\em On the convergence of a fourth-order method for a class of singular
	boundary value problems},
	\newblock Journal of Computational and Applied Mathematics. 224(2009), 734 --
	742.
	
	\bibitem{pandey2008existence}
	R. K. Pandey and A.K. Verma,
	\newblock{\em Existence-uniqueness results for a class of singular boundary value
	problems arising in physiology},
	\newblock Nonlinear Analysis: Real World Applications. 9(2008), 40 -- 52.
	
	\bibitem{Pandey2008existence1}
	R. K. Pandey and A.K. Verma,
	\newblock{\em Existence-uniqueness results for a class of singular boundary value
	problems-ii},
	\newblock Journal of Mathematical Analysis and Applications. 338(2008), 1387-- 1396.
	
	\bibitem{pandey2009note}
	R. K. Pandey and A.K. Verma,
	\newblock{\em A note on existence-uniqueness results for a class of doubly singular
	boundary value problems},
	\newblock Nonlinear Analysis: Theory, Methods \& Applications. 71(2009), 3477-- 3487.
	
	\bibitem{MCPLAW2012}
	M. C. Pereyra and L.A. Ward,
	\newblock{\em Harmonic Analysis: From Fourier to Wavelets},
	\newblock Student mathematical library, 2012.
	
	\bibitem{RR2015}
	R.~Rajaraman and G.~Hariharan.
	\newblock{\em An efficient wavelet based spectral method to singular boundary value
	problems},
	\newblock Journal of Mathematical Chemistry. 53(2015), 2095 -- 2113.
	
	\bibitem{Saeed2014}
	U. Saeed and M. Ur~Rehman,
	\newblock{\em Hermite wavelet method for fractional delay differential equations},
	\newblock Journal of Difference Equations. 2014(2014), 359093.
	
	\bibitem{SC2016}
	S.~C. Shiralashetti, A.B. Deshi, and P.B.~Mutalik Desai,
	\newblock{\em Haar wavelet collocation method for the numerical solution of
	singular initial value problems},
	\newblock Ain Shams Engineering Journal. 7(2016), 663--670.
	
	\bibitem{Mandeep-Amit-2016}
	M. Singh and A.~K. Verma,
	\newblock{\em An effective computational technique for a class of lane-emden
	equations},
	\newblock Journal of Mathematical Chemistry. 54(2016), 231--251.
	
	\bibitem{rsgjshgag2019}
	R.~Singh, J.~Shahni, H.~Garg, and A.~Garg,
	\newblock{\em Haar wavelet collocation approach for lane-emden equations arising in
	mathematical physics and astrophysics},
	\newblock The European Physical Journal Plus. 134(2019).
	
	\bibitem{rsg2019}
	R. Singh, H. Garg, and V. Guleria,
	\newblock{\em Haar wavelet collocation method for lane-emden equations with
	dirichlet, neumann and neumann-robin boundary conditions},
	\newblock Journal of Computational and Applied Mathematics. 346(2019), 150 --
	161.
	
	\bibitem{Mujeeb2015}
	M. Ur Rehman and U. Saeed,
	\newblock{\em Gegenbauer wavelets operational matrix method for fractional
	differential equations},
	\newblock Journal of the Korean Mathematical Society. 52(2015), 1069--1096.
	
	\bibitem{Usman2013}
	M. Usman and S.~T. Mohyud-Din,
	\newblock{\em Physicists hermite wavelet method for singular differential
	equations},
	\newblock International Journal of Advances in Applied Mathematics and
	Mechanics. 1(2013), 16--29.
	
	\bibitem{akvsbvpreview2020}
	A.~K. Verma, B.~Pandit, L.~Verma, and R.~P. Agarwal,
	\newblock{\em A review on a class of second order nonlinear singular {BVPs}},
	\newblock Mathematics. 8(2020), 1045.
	
	\bibitem{skcoam2020}
	A.~K. Verma and S. Kayenat,
	\newblock{\em Applications of modified {Mickens-type} {NSFD} schemes to
	{Lane-Emden} equations},
	\newblock Computational and Applied Mathematics. 39(2020), 227.
	
	\bibitem{akvdt2019}
	A.~K. Verma and D. Tiwari,
	\newblock{\em Higher resolution methods based on quasilinearization and haar
	wavelets on lane-emden equations},
	\newblock International Journal of Wavelets, Multiresolution and Information Processing. 17(2019), 1950005.
	
	\bibitem{Vijesh2016}
	V.~A. Vijesh, L.~A. Sunny, and K.~H. Kumar,
	\newblock{\em Legendre wavelet quasilinearization technique for solving
	q-difference equations},
	\newblock Journal of Difference Equations and Applications.
	22(2016), 594--606.
	
	\bibitem{Zhou2016}
	F. Zhou and X. Xu,
	\newblock{\em Numerical solutions for the linear and nonlinear singular boundary
	value problems using laguerre wavelets},
	\newblock Advances in Difference Equations. 2016(2016), 17.
	
	\bibitem{MAJAK2015321}
	J. Majak, B. Shvartsman, K. Karjust, M. Mikola, A. Haavajoe and M. Pohlak,
	\newblock{\em On the accuracy of the Haar wavelet discretization method},
	\newblock Composites Part B: Engineering. 80(2015), 321--327.

   \bibitem{Sunil2014}
	S. Kumar,
    \newblock A new analytical modelling for fractional telegraph equation via Laplace transform,
    \newblock{\em Applied Mathematical Modelling}. 38(2014), 3154--3163.

    \bibitem{Rashidi2014}
	S. Kumar and M. M. Rashidi,
    \newblock{\em New analytical method for gas dynamics equation arising in shock fronts},
    \newblock Computer Physics Communications. 185(2014), 1947-1954.

    \bibitem{Ranbir2020}
    B. Ghanbari, S. Kumar and R. Kumar,
    \newblock{\em A study of behaviour for immune and tumor cells in immunogenetic tumour model with non-singular fractional derivative},
    \newblock Chaos, Solitons \& Fractals. 133(2020), 109619.
    
   \bibitem{Mugisha2020}
    E. F. D. Goufo, S. Kumar and S.B. Mugisha,
    \newblock{\em Similarities in a fifth-order evolution equation with and with no singular kernel},
    \newblock Chaos, Solitons \& Fractals. 130(2020), 109467.

     \bibitem{Bessem2020}
     S. Kumar, R. Kumar, R.P. Agarwal and B. Samet,  
     \newblock{\em  A study of fractional Lotka-Volterra population model using Haar wavelet and Adams-Bashforth-Moulton methods},
     \newblock Mathematical Methods in the Applied Sciences. 43(2020), 5564-5578.
 
     \bibitem{Manafian2020}
     S. Pourghanbar, J. Manafian, M. Ranjbar, A. Aliyeva and Y. S. Gasimov,
     \newblock{\em  An Efficient Alternating Direction Explicit Method for Solving a Nonlinear Partial Differential Equation},
     \newblock Mathematical Problems in Engineering. 2020(2020), 9647416.

    \bibitem{Hossein2020}
     N. Can, O. Nikan, M. Rasoulizadeh, H. Jafari and Y. Gasimov,
     \newblock{\em  Numerical computation of the time non-linear fractional generalized equal width model arising in shallow water channel}, 
     \newblock Thermal Science. 24(2020), 49-58.

    \bibitem{Fatima2020}
     F. Aboud and A. Nachaoui,
     \newblock{\em  Single-rank Quasi-Newton methods for the solution of nonlinear semiconductor equations}, 
     \newblock Advanced Mathematical Models \& Applications. 5(2020), 70-79.
 
    \bibitem{Aziz2020}
     I. Aziz and Q. U. Ain,
     \newblock{\em  Numerical solution of partial integro-differential equations with weakly singular kernels},
     \newblock Advanced Mathematical Models \& Applications. 5(2020), 149-160.
     
 
	
\end{thebibliography}
\end{document}